\documentclass[11pt]{article}

\usepackage{amsmath}
\usepackage{amssymb}
\usepackage{amsthm}
\usepackage{amsfonts}
\usepackage{graphicx}
\usepackage{textcomp}
\usepackage{color}
\usepackage[hidelinks]{hyperref}
\usepackage{multicol}
\usepackage{slashed}
\usepackage{tikz-cd}
\usepackage{setspace}
\usepackage[utf8]{inputenc}
\usepackage[english]{babel}
\usepackage{mathrsfs}
\usepackage[a4paper, total={5.7in, 9.5in}]{geometry}
\let\OLDitemize\itemize
\renewcommand\itemize{\OLDitemize\addtolength{\itemsep}{-8pt}}

\numberwithin{equation}{section}

\newtheorem{Proposition}{Proposition}[section]
\newtheorem{Lemma}{Lemma}[section]
\newtheorem{Theorem}{Theorem}[section]
\newtheorem{Corollary}{Corollary}[section]
\newtheorem{Definition}{Definition}[section]

\begin{document}

\begin{titlepage}

\title{Pre-Metric-Bourbaki Algebroids: Cartan Calculus for M-Theory}

\author{Aybike \c{C}atal-\"{O}zer$^1$, Tekin Dereli$^{2, 3}$, Keremcan Do\u{g}an$^{1 \dagger}$ \\ \small{$^1$Department of Mathematics, \.{I}stanbul Technical University, \.{I}stanbul, Turkey,} \\ \small{$^2$Department of Physics, Ko\c{c} University, \.{I}stanbul, Turkey} \\ \small{$^3$Faculty of Engineering and Natural Sciences, Maltepe University, \.{I}stanbul, Turkey}}

\date{$^{\dagger}$Corresponding author, \mbox{E-mail: dogankerem[at]itu.edu.tr}}

\maketitle

\begin{abstract}
\noindent String and M theories seem to require generalizations of usual notions of differential geometry on smooth manifolds. Such generalizations usually involve extending the tangent bundle to larger vector bundles equipped with various algebroid structures such as Courant algebroids, higher Courant algebroids, metric algebroids, or $G$-algebroids. The most general geometric scheme is not well understood yet, and a unifying framework for such algebroid structures is needed. Our aim in this paper is to propose such a general framework. Our strategy is to follow the hierarchy of defining axioms for a Courant algebroid: almost-Courant - metric - pre-Courant - Courant. In particular, we focus on the symmetric part of the bracket and the metric invariance property, and try to make sense of them in a manner as general as possible. These ideas lead us to define new algebroid structures which we dub Bourbaki and metric-Bourbaki algebroids, together with their almost- and pre-versions. For a special case of pre-metric-Bourbaki algebroids that we call exact, we construct a collection of maps which generalize the Cartan calculus of exterior derivative, Lie derivative and interior product. This is done by a kind of reverse-mathematical analysis of the \v{S}evera classification of exact Courant algebroids. By abstracting crucial properties of this collection of maps, we define the notion of Bourbaki pre-calculus. Conversely, given an arbitrary Bourbaki pre-calculus, we construct a pre-metric-Bourbaki algebroid by building up a standard bracket that is analogous to the Dorfman bracket. Moreover, we prove that any exact pre-metric-Bourbaki algebroid satisfying some further conditions has to have a bracket that is the twisted version of the standard bracket; a partly analogous result to \v{S}evera classification. We prove that many physically and mathematically motivated algebroids from the literature are examples of these new algebroids, and when possible we construct a Bourbaki pre-calculus on them. In particular, we show that the Cartan calculus can be seen as the Bourbaki pre-calculus corresponding to an exact higher pre-Courant algebroid. We also point out examples of Bourbaki pre-calculi including the generalization of the Cartan calculus on vector bundle valued forms. One straightforward generalization of our constructions might be done by replacing the tangent bundle with an arbitrary Lie algebroid $A$. This step allows us to define $A$-Bourbaki algebroids and Bourbaki $A$-pre-calculus, and extend our main results for them while proving many other algebroids from the literature fit into this framework.

\end{abstract}

\vskip 2cm



\thispagestyle{empty}

\end{titlepage}

\maketitle


\section{Introduction}
\label{s1}

\noindent Geometry and physics have always been in a profound relation. Historically, Newtonian classical mechanics depends on the Euclidean geometry of the absolute space. On the other hand, Einstein's relativity forces one to work in the broader framework of (semi-)Riemannian geometry. Mathematical structures of this geometry are constructed on the tangent bundle $T(M)$ of a smooth manifold $M$, and the Einstein-Hilbert action 
\begin{equation}
S_{EH}[g] = \int_M{R(^g \nabla, g) *_g 1}
\label{ea}
\end{equation}
leads to Einstein field equations, where $R(^g \nabla, g)$ is the Ricci scalar of a metric $g$ and its associated Levi-Civita connection $^g \nabla$. These equations explain gravity in terms of differential geometric structures. Yet, a huge part of fundamental physics, namely quantum field theory, is inconsistent with general relativity. String and M theories are arguably one of the most promising candidates in order to solve this inconsistency. In order to explain some of the features of string theories, it seems necessary to consider geometric structures that suppress the usual notions of differential geometry on smooth manifolds. For example, bosonic sector of string theories can be formulated in the framework of double field theory, which geometrizes Kalb-Ramond $B$-field and $T$-duality \cite{1}:
\begin{equation}
S[g, H, \phi] = \int_M {e^{-2 \phi} \left[ R(^g \nabla, g) *_g 1 - \frac{1}{2} H \wedge *_g H + 4 d \phi \wedge *_g d \phi \right]},
\label{ea2}
\end{equation}
where $H$ is the field strength of the Kalb-Ramond $B$-field, and $\phi$ is the dilaton. This theory is closely related to generalized geometry \cite{2} on exact Courant algebroids, which are of the form $T(M) \oplus T^*(M)$, where $T^*(M)$ is the cotangent bundle of $M$. Consequently, the action (\ref{ea2}) can be written in a geometric sense in terms of generalized or ``doubled'' geometric structures; specifically by using the generalized Ricci scalar. For other sectors and $U$-duality, one can use exceptional generalized geometry which is related to higher Courant algebroids \cite{3} and $G$-algebroids \cite{4}. There are many other both physically and mathematically motivated examples including $E$-Courant \cite{5}, $AV$-Courant \cite{6} and omni-Lie \cite{7} algebroids, all of which is a generalization of Courant algebroids in some certain sense. The tangent bundle $T(M)$ itself is a Lie algebroid with the Lie bracket of vector fields, and one can generalize the usual differential geometry on $T(M)$ to an arbitrary Lie algebroid \cite{8}. 

As these algebroid structures are widely used in the literature in order to find generalizations of a differential geometry suitable for string and M theories, establishing a more unifying framework of algebroids is crucial. The aim of this paper is to establish such a general framework, while our strategy is to focus on the hierarchy of axioms leading to a Courant algebroid, and to analyze the \v{S}evera classification of exact Courant algebroids in details. Namely, we consider the relations between almost-Courant, metric, pre-Courant and Courant algebroids, with a particular interest in the symmetric part of the bracket and the metric invariance property. We start with the observation that any local almost-Leibniz bracket's symmetric part has to satisfy a certain condition. By using this condition, we decompose the symmetric part into two components: a vector bundle valued metric and a first-order differential operator. This leads us to the definition of almost-Bourbaki algebroids\footnote{The name ``Bourbaki'' comes from an inside joke; when KD declared a wish about ``constructing a generalization of $E$-Courant algebroids which generalize Courant algebroids'', AÇÖ compared his mentality to Bourbaki tradition, and in the next meeting KD defined Bourbaki algebroids, satirically. We found no better alternative, so we stick with the name Bourbaki. Coincidentally, in one of our main references \cite{9} about first-order differential operators, Bourbaki's book on algebra \cite{10} is cited for a discussion of a general notion of derivations.}. Then, by examining the metric invariance property, we consider a derivation bundle valued map, and define almost-metric-Bourbaki algebroids which satisfy a metric invariance condition in some generalized sense. By considering the Leibniz-Jacobi identity and a certain morphism property of the anchor, we also define Bourbaki, metric-Bourbaki algebroids and their pre-versions. A special case of Bourbaki algebroids that we call exact shed some lights on the geometry on the direct sum $T(M) \oplus Z$ for some arbitrary vector bundle $Z$. By using exact pre-metric-Bourbaki algebroids, we construct a collection of maps which are analogous to exterior derivative, Lie derivative and interior product. By abstracting certain crucial properties of these maps, we define the notion of Bourbaki pre-calculus as a generalization of the Cartan calculus. We investigate the relation between Bourbaki algebroids and Bourbaki pre-calculus thoroughly. We show that given an arbitrary Bourbaki pre-calculus, one can construct a pre-metric-Bourbaki algebroid, and for certain exact pre-metric algebroids, the bracket has to be a twisted version of a Dorfman-like bracket; a result partly partly extending the \v{S}evera classification. As a natural extension, by replacing the tangent bundle with an arbitrary Lie algebroid $A$, we also define $A$-Bourbaki algebroids and Bourbaki $A$-pre-calculus completely analogously, and extend our main results for them. Furthermore, we give several examples of algebroids and calculi for both the usual and $A$-versions.

As we use many prefixes in our paper, a warning should be made: the definitions of algebroids, especially the ones with ``almost'' or ``pre'' prefixes, might differ from paper to paper\footnote{As the names of the structures in the algebroid/algebra literature have already gotten out of hand, we use this opportunity to give \textit{the best prefix award} to Weinstein for omni-Lie algebras \cite{11}.}. In the text, we use ``almost'' when there is only one (right-)Leibniz rule. We add the term ``local'', when the second (left-)Leibniz rule is also introduced. ``Pre'' automatically means that the anchor is a morphism of brackets. The term ``Lie'' appears when the bracket is anti-symmetric. If there is no prefix, the algebroid often satisfies the Leibniz-Jacobi identity. Moreover, if there is the adjective ``exact'', the underlying vector bundle fits into a certain short exact sequence of vector bundles. We also use the word ``metric'' as a prefix, and in this case a generalization of the metric invariance property is assumed.

In Section (\ref{s2}), we summarize the definitions and important properties of interior product, exterior derivative, and Lie derivative. In Section (\ref{s3}), we discuss anchored vector bundles and splittings of short exact sequences in the category of vector bundles. In Section (\ref{s4}), we outline the hierarchy of almost-Courant, metric, pre-Courant and Courant algebroids, and comment on the \v{S}evera classification of exact Courant algebroids. Sections (\ref{s2}, \ref{s3}, \ref{s4}) are introductory, and they are included for pedagogical purposes, so they might be skipped by more experienced readers. Yet, they also serve as a summary of the notation of the paper. In Section (\ref{s5}), we first discuss ``standard'' higher Courant algebroids as they appear in the existing literature. Then, we axiomatically define higher Courant algebroids, which will provide the main insight for the consecutive sections. Sections (\ref{s6} - \ref{s9}) constitute the main and original part of our work. In Section (\ref{s6}), we examine the hierarchy of Courant algebroids starting from the symmetric part of the bracket, and we define Bourbaki algebroids in terms of a vector bundle valued metric and a first-order differential operator. In Section (\ref{s7}), we focus on the metric invariance property, and define metric-Bourbaki algebroids. In Section (\ref{s8}), we define exact versions of these algebroids. In Section (\ref{s9}), we introduce Bourbaki pre-calculus, analyze the relation between this calculus and Bourbaki algebroids, and present our main theorems. In Section (\ref{s10}), we present many examples of Bourbaki and metric-Bourbaki algebroids together with Bourbaki pre-calculi. In Section (\ref{s11}), we generalize our constructions and results once more, and define $A$-Bourbaki algebroids and Bourbaki $A$-pre-calculi for an arbitrary Lie algebroid $A$, where in Section (\ref{s12}), we give several examples of these structures. Sections (\ref{s10}) and (\ref{s12}) contain a large number of algebroid examples from the literature, so they serve as a dictionary of algebroids with a unified notation. In Section (\ref{s13}), we make some concluding remarks, and point out certain possible future research routes.


\section{Cartan Calculus on Smooth Manifolds}
\label{s2}

\noindent In this section, we set our notation and summarize the mathematical structures on the tangent bundle that will be used throughout the paper. In particular, we focus on Cartan calculus, namely the exterior derivative, Lie derivative and interior product, and certain relations between them.

For a smooth manifold $M$, its tangent bundle is denoted by $T(M)$. Sections of the tangent bundle are the vector fields. The set of vector fields, $\mathfrak{X}(M)$, is isomorphic to the derivations on the ring of smooth functions $C^{\infty}(M, \mathbb{R})$, and the action of a vector field $V$ on a smooth function $f$ is denoted by $V(f)$. The set of vector fields becomes a Lie algebra with the Lie bracket, which is an anti-symmetric, $\mathbb{R}$-bilinear map $[\cdot,\cdot]: \mathfrak{X}(M) \times \mathfrak{X}(M) \to \mathfrak{X}(M)$ that satisfies the Jacobi identity. Moreover, as the Lie bracket also satisfies the Leibniz rule, the triplet $(T(M), id_{T(M)}, [\cdot,\cdot])$ becomes a Lie algebroid called the tangent Lie algebroid, where $id_X$ denotes the identity map for a set $X$. The dual of the tangent bundle is the cotangent bundle $T^*(M)$, and its sections are (exterior differential) $1$-forms whose set is denoted by $\Omega^1(M)$. By using the tensor products of the tangent bundle and the cotangent bundle, one can construct the set of tensors of $(q, r)$-type, which is denoted by $Tens^{(q, r)}(M)$. The Lie bracket of vector fields can be extended to the Lie derivative acting on the whole tensor algebra by the Leibniz rule, i.e. the Lie derivative with respect to a vector field $V$ is a map 
$\mathcal{L}_V: Tens^{(q, r)}(M) \to Tens^{(q, r)}(M)$ satisfying
\begin{align} 
(\mathcal{L}_V T)(V_1, \ldots, V_r, \omega_1, \ldots, \omega_q) &:= V(T(V_1, \ldots, V_r, \omega_1, \ldots, \omega_q)) \nonumber\\
& \quad \ - \sum_{i=1}^r T(V_1, \ldots, \mathcal{L}_V V_i, \ldots V_r, \omega_1, \ldots, \omega_q) \nonumber\\
& \quad \ - \sum_{i=1}^q T(V_1, \ldots, V_r, \omega_1, \ldots, \mathcal{L}_V \omega_i, \ldots, \omega_q),
\label{eb1}
\end{align}
for $T \in Tens^{(q, r)}(M), V, V_i \in \mathfrak{X}(M), \omega_i \in \Omega^1(M)$. In particular, one has
\begin{align} 
\mathcal{L}_V f &= V(f), \nonumber\\
\mathcal{L}_V U &= [V, U], \nonumber\\
(\mathcal{L}_V \omega)(U) &= V(\omega(U)) - \omega([V, U]),
\label{eb2}
\end{align}
for all $f \in C^{\infty}(M, \mathbb{R}), U, V \in \mathfrak{X}(M), \omega \in \Omega^1(M)$. Anti-symmetric $(0, p)$-type tensors are called $p$-forms, and the set of $p$-forms is denoted by $\Omega^p(M)$. The Lie derivative acting on $p$-forms can be seen as a map $\mathcal{L}_V: \Omega^p(M) \to \Omega^p(M)$, and it is a graded derivation of degree 0 on the set of all forms $\Omega(M) := \oplus_{i=1}^{dim(M)} \Omega^i(M)$. There is also a graded derivation of degree 1 on $\Omega(M)$ called the exterior derivative which is defined as a map $d: \Omega^p(M) \to \Omega^{p+1}(M)$ such that
\begin{align}
(d\omega)(V_1, \ldots, V_{p+1}) &:= \sum_{1 \leq i \leq p+1} (-1)^{i+1} V_i(\omega(V_1, \ldots, \check{V_i}, \ldots, V_{p+1})) \nonumber\\
& \quad \ + \sum_{1 \leq i < j \leq {p+1}} \omega([V_i, V_j], V_1, \ldots, \check{V_i}, \ldots \check{V_j}, \ldots, V_{p+1}),
\label{eb3}
\end{align}
for $V_i \in \mathfrak{X}(M)$, where $\check{V}$ indicates that $V$ is excluded. As the square of the exterior derivative is 0 by the Jacobi identity of the Lie bracket, $(\Omega(M), d)$ becomes a differential graded algebra. There is also a family of differential graded algebra structures induced by the interior product with respect to vector fields. The interior product with respect to a vector field $V$ is defined as a map $\iota_V: \Omega^{p}(M) \to \Omega^{p-1}(M)$ such that
\begin{equation} 
(\iota_V \omega)(V_1, \ldots, V_{p-1}) := \omega(V, V_1, \ldots, V_{p-1}),
\label{eb4}
\end{equation}
for $V_i \in \mathfrak{X}(M)$. The anti-symmetry of $p$-forms implies that $\iota_V^2 = 0$ for all $V \in \mathfrak{X}(M)$, and it is a graded derivation of degree $-1$. The exterior derivative, interior product and the Lie derivative satisfy
\begin{align}
d^2 &= \iota_V^2 = 0, 
\label{eb5} \\
\mathcal{L}_V &= d \iota_V + \iota_V d,
\label{eb6} \\
\mathcal{L}_U \iota_V &= \iota_{[U, V]} + \iota_V \mathcal{L}_U,
\label{eb7} \\
\mathcal{L}_{[U, V]} &= \mathcal{L}_U \mathcal{L}_V - \mathcal{L}_V \mathcal{L}_U, 
\label{eb8} \\
\mathcal{L}_V d &= d \mathcal{L}_V,
\label{eb9}
\end{align}
for all $U, V \in \mathfrak{X}(M)$. Equation (\ref{eb8}) follows from the Jacobi identity for the Lie bracket, and Equation (\ref{eb9}) follows from $d^2 = 0$ and Cartan magic formula (\ref{eb6}). One can also prove the following relations about the $C^{\infty}(M, \mathbb{R})$-module structure:
\begin{align}
\iota_{f V} \omega &= \iota_V (f \omega) = f \iota_V \omega, \nonumber\\
d(f \omega) &= f d \omega + d f \wedge \omega, \nonumber\\
\mathcal{L}_V (f \omega) &= f \mathcal{L}_V \omega + V(f) \omega, \nonumber\\
\mathcal{L}_{f V} \omega &= f \mathcal{L}_V \omega + d f \wedge \iota_V \omega, 
\label{eb10}
\end{align}
for all $f \in C^{\infty}(M, \mathbb{R}), V \in \mathfrak{X}(M), \omega \in \Omega^p(M)$.

By Cartan calculus on a smooth manifold, we mean the study of interior product, exterior derivative, Lie derivative, and the relations between them given in this section. More formally, Cartan calculus on $M$ is a triplet $(\iota, d, \mathcal{L})$ from $Der(\Omega(M)) := \oplus_k Der^k (\Omega(M))$, where $Der^k(\Omega(M))$ is the set of derivations on $\Omega(M)$ of degree $k$, such that Equations (\ref{eb5}) - (\ref{eb9}) hold.

We finish this section with some remarks about the notation of the paper. First of all, everything is assumed to be in the smooth category. We only consider real and finite-rank vector bundles over the \textit{same} base manifold $M$, which is Hausdorff and paracompact. Moreover, every map between vector bundles is assumed to be covering the identity $id_M$. For a vector bundle $E$, the set of its sections is denoted by $\Gamma(E)$. The composition of two maps $\xi_1, \xi_2$ with appropriate domains and ranges is denoted by $\xi_1 \xi_2$ without anything in between. We do not use anything for the scalar multiplication for a module structure, either. The only modules in this paper are modules over the ring of smooth functions, so this notation only appears with smooth functions. 



\section{Anchored Vector Bundles}
\label{s3}

\noindent In this section, we continue to set the notation of the paper while explaining anchored vector bundles. We try to generalize some of the constructions of the previous section, including tensors, vector fields, forms, metrics, connections and interior product, for an arbitrary vector bundle equipped with an anchor.

One can generalize the definition of tensors on a manifold to $E$-tensors on an arbitrary (real, finite-rank) vector bundle $E$, where a $(q, r)$-type $E$-tensor is defined as an element of
\begin{equation}
Tens^{(q, r)}(E) := \Gamma \left( \bigotimes_{i = 1}^q E \otimes \bigotimes_{j = 1}^r E^* \right).
\label{ec1}
\end{equation}
Similarly, one can define $E$-vector fields, $\mathfrak{X}(E) := \Gamma(E)$, $E$-1-forms $\Omega^1(E) := \Gamma(E^*)$, and $E$-$p$-forms, $\Omega^p(E)$, local $E$-frames, local $E$-coframes, completely analogously. On the set of all $E$-$p$-forms, one can define an $E$-interior product $\iota^E_v: \Omega^p(E) \to \Omega^{p-1}(E)$ analogously to Equation (\ref{eb4}). Consequently, it also squares to 0, i.e. $\left( \iota^E_v \right)^2 = 0$ for all $v \in \mathfrak{X}(E)$ due to anti-symmetry. 

A fiber-wise metric on $E$ can be considered as a non-degenerate, symmetric $(0, 2)$-type $E$-tensor, and any such map is called an $E$-metric. Any $E$-metric $g: \mathfrak{X}(E) \times \mathfrak{X}(E) \to C^{\infty}(M, \mathbb{R})$ induces a map, the \textit{flat} musical isomorphism, $g: \mathfrak{X}(E) \to \Omega^1(E)$, which is denoted by the same letter. As an $E$-metric is non-degenerate, one can define the inverse $E$-metric $g^{-1}: \Omega^1(E) \times \Omega^1(E) \to C^{\infty}(M, \mathbb{R})$. Similarly, the inverse $E$-metric also induces the \textit{sharp} musical isomorphism $g^{-1}: \Omega^1(E) \to \mathfrak{X}(E)$. One can also define vector bundle valued $E$-tensors. In particular, for a vector bundle $R$, an $R$-valued $E$-metric is defined as a symmetric, non-degenerate, $C^{\infty}(M, \mathbb{R})$-bilinear map $g: \mathfrak{X}(E) \times \mathfrak{X}(E) \to \mathfrak{X}(R)$. 

As the sections of a general vector bundle $E$ do not act as derivations on the smooth functions, one often needs to introduce an anchor, which is a vector bundle morphism from the vector bundle $E$ to the tangent bundle whose sections are derivations on smooth functions, for defining many geometrical objects including $E$-connections. 

\begin{Definition} A doublet $(E, \rho_E)$ is called an anchored vector bundle if $\rho_E: E \to T(M)$ is a vector bundle morphism.
\label{dc1}
\end{Definition}
As vector bundle morphisms are $C^{\infty}(M, \mathbb{R})$-linear, they induce a map on the sections. We use the same notation for these maps, so in particular we work with a map $\rho_E: = \mathfrak{X}(E) \to \mathfrak{X}(M)$. The anchor $\rho_E: E \to T(M)$ also induces another map $D_E := \rho_E^* d: C^{\infty}(M, \mathbb{R}) \to \Omega^1(E)$, where $d$ is the usual exterior derivative, and $^*$ denotes the transpose of a vector bundle morphism. In other words, $(D_E f)(u) = \rho_E(u)(f)$, for all $f \in C^{\infty}(M, \mathbb{R}), u \in \mathfrak{X}(E)$.

By using the anchor, one can define a generalization of vector bundle connections.
\begin{Definition} \cite{12} Let $(E, \rho_E)$ be an anchored vector bundle, and $R$ a vector bundle. An $E$-connection on $R$ is an $\mathbb{R}$-bilinear map $\nabla: \mathfrak{X}(E) \times \mathfrak{X}(R) \to \mathfrak{X}(R)$ such that
\begin{align}
    \nabla_v (f r) &= f \nabla_v r + \rho_E(v)(f) r, \nonumber\\
    \nabla_{f v} r &= f \nabla_v r,
\label{ec2}
\end{align}
for all $v \in \mathfrak{X}(E), f \in C^{\infty}(M, \mathbb{R}), r \in \mathfrak{X}(R)$. An $E$-connection on $E$ is called a linear $E$-connection. 
\label{dc2}
\end{Definition}
\noindent Note that any $E$-connection induces a map $\nabla: \mathfrak{X}(R) \to \mathfrak{X}(E^*) \times \mathfrak{X}(R)$.

If the anchor is surjective, then an anchored vector bundle is called \textit{transitive}. Moreover, if the anchor is of locally constant rank, then such a vector bundle is called \textit{regular}. For regular vector bundles, the kernel of the anchor is a subbundle, and every transitive vector bundle is regular. Therefore, for transitive vector bundles, one has the following short exact sequence in the category of vector bundles:

\begin{equation}
    0 \xrightarrow{\quad} ker(\rho_E) \xrightarrow{\quad inc \quad} E \xrightarrow{\ \quad \rho_E \ \quad} T(M) \xrightarrow{\quad} 0,
    \label{ec3}
\end{equation}
where $inc: ker(\rho_E) \to E$ is the canonical inclusion map from the kernel of the anchor to $E$. Many algebroid examples from the literature fits into a short exact sequence of vector bundles of the form

\begin{equation}
    0 \xrightarrow{\quad} Z \xrightarrow{\quad \ \chi \ \quad} E \xrightarrow{\ \quad \rho_E \ \quad} T(M) \xrightarrow{\quad} 0,
    \label{ec4}
\end{equation}
for some vector bundle $Z$ and vector bundle morphism $\chi: Z \to E$. For instance, for exact Courant algebroids $Z = T^*(M)$, and for exact higher Courant algebroids $Z = \Lambda^{p-1}(T^*(M))$. Motivated from these examples, for exact Bourbaki algebroids of Section (\ref{s8}), we will focus on the general properties of the vector bundle $Z$. 

As we consider real and finite-rank vector bundles over a paracompact manifold, any short exact sequence of vector bundles splits, i.e. there is a vector bundle morphism $\phi: T(M) \to E$ such that $\rho_E \phi = id_{T(M)}$. The (right-)splitting $\phi$ induces a vector bundle isomorphism $T(M) \oplus Z \to E$ given by $\phi \oplus \chi$, where $(\phi \oplus \chi)(U + z) := \phi(U) + \chi(z)$ for $U \in \mathfrak{X}(M), z \in \mathfrak{X}(Z)$. When $E$ is equipped with an $E$-metric $g$ (possibly vector bundle valued), a map whose range is in $E$, e.g., a splitting $\phi: T(M) \to E$, is called $g$-isotropic if its image is isotropic with respect to $g$, i.e.
\begin{equation}
    g(\phi(U), \phi(V)) = 0,
\label{ec5}
\end{equation}
for all $U, V \in \mathfrak{X}(M)$. 

As we did in Section (\ref{s2}), we finish this section with some notation related remarks. Although we always explicitly state in the text, uppercase Latin letters such as $U, V, W$ denote usual vector fields on a manifold, and lowercase Greek letters such as $\omega, \eta$ denote usual $p$-forms. On the other hand, lowercase Latin letters such as $u, v, w$ denote sections of a general vector bundle $E$ (i.e. $E$-vector fields), and uppercase Greek letters such as $\Omega$ denote $E$-$p$-forms. When we use two or more vector bundles together, we use the notation $u, v, w$ for the sections of $E$, and for the others we use the neighborhood of the lowercase version of the vector bundle; for example $r, q, p$ for the sections of a vector bundle $R$ or $z, y, x$ for $Z$.


\section{Metric and Courant Algebroids}
\label{s4}

\noindent In this section, we present the definitions of almost-Courant, metric, pre-Courant and Courant algebroids, all of which fit into a hierarchy of axioms. We start the constructions from the standard Courant algebroid equipped with the Dorfman bracket.

The \textit{standard Courant algebroid} is of the form $\mathbb{T}(M) := T(M) \oplus T^*(M)$ which is equipped with the \textit{Dorfman bracket}:
\begin{equation}
    [U + \omega, V + \eta]_D := [U, V] + \mathcal{L}_U \eta - \mathcal{L}_V \omega + d \iota_V \omega,
\label{ed1}
\end{equation}
where $U, V \in \mathfrak{X}(M), \omega, \eta \in \Omega^1(M)$, so that $U + \omega, V + \eta \in \mathfrak{X}(T(M) \oplus T^*(M)) = \mathfrak{X}(\mathbb{T}(M))$. The Dorfman bracket satisfies the following properties
\begin{align}
    [u, f v]_D &= f [u, v]_D + proj_1(u)(f) v, \nonumber\\
    [u, v]_D + [v, u]_D &= g_S^{-1} D_{\mathbb{T}(M)} g_S(u, v), \nonumber\\
    proj_1(u)(g_S(v, w)) &= g_S([u, v]_D, w) + g_S(v, [u,w]_D), \nonumber\\
    proj_1([u, v]_D) &= [proj_1(u), proj_1(v)], \nonumber\\
    [u, [v, w]_D]_D &= [[u, v]_D, w]_D + [v, [u, w]_D]_D,
\label{ed2}
\end{align}
for all $u, v, w \in \mathfrak{X}(\mathbb{T}(M)), f \in C^{\infty}(M, \mathbb{R})$, where $g_S: \mathfrak{X}(\mathbb{T}(M)) \times \mathfrak{X}(\mathbb{T}(M)) \to C^{\infty}(M, \mathbb{R})$ is defined by
\begin{equation}
    g_S(U + \omega, V + \eta) := \iota_U \eta + \iota_V \omega,
\label{ed3}
\end{equation} 
$proj_1: \mathbb{T}(M) = T(M) \oplus T^*(M) \to T(M)$ is the projection onto the first component, and $D_{\mathbb{T}(M)} := proj_1^* d$. Note that $(\mathbb{T}(M) = T(M) \oplus T^*(M), \rho_E = proj_1)$ is an anchored vector bundle, and $g_S$ is $C^{\infty}(M, \mathbb{R})$-bilinear, symmetric and non-degenerate, so it is a $\mathbb{T}(M)$-metric. By abstracting some or all of these properties, one can get the notion of many algebroids.
\begin{Definition} A triplet  $(E, \rho_E, [\cdot,\cdot]_E)$ is called an almost-Leibniz algebroid if
\begin{itemize}
    \item $(E, \rho_E)$ is an anchored vector bundle,
    \item $[\cdot,\cdot]_E: \mathfrak{X}(E) \times \mathfrak{X}(E) \to \mathfrak{X}(E)$ is an $\mathbb{R}$-bilinear map satisfying the right-Leibniz rule, i.e.
    \begin{equation}
        [u, f v]_E = f [u, v]_E + \rho_E(u)(f) v,
    \label{ed4}
    \end{equation}
    for all $u, v \in \mathfrak{X}(E), f \in C^{\infty}(M, \mathbb{R})$.
\end{itemize}
An almost-Leibniz algebroid is called a pre-Leibniz algebroid if the anchor is a morphism of brackets, i.e.
\begin{equation}
    \rho_E([u, v]_E) = [\rho_E(u), \rho_E(v)],
\label{ed5}
\end{equation}
for all $u, v \in \mathfrak{X}(E)$. Moreover, an almost-Leibniz algebroid is called a Leibniz algebroid if the Leibniz-Jacobi identity holds, i.e.
\begin{equation}
    [u, [v, w]_E]_E = [[u, v]_E, w]_E + [v, [u, w]_E]_E,
\label{ed6}
\end{equation}
for all $u, v, w \in \mathfrak{X}(E)$.
\label{dd1}
\end{Definition}
\noindent As the Leibniz-Jacobi identity (\ref{ed6}) and the right-Leibniz rule (\ref{ed4}) together imply that the anchor is a morphism of brackets, every Leibniz algebroid is a pre-Leibniz algebroid.

\begin{Definition} \cite{13} A quadruplet  $(E, \rho_E, [\cdot,\cdot]_E, g)$ is called an almost-Courant algebroid if
\begin{itemize}
    \item $(E, \rho_E, [\cdot,\cdot]_E)$ is an almost-Leibniz algebroid,
    \item $g$ is an $E$-metric,
    \item The symmetric part of the bracket\footnote{When we say ``the symmetric part of a bracket'', we will always ignore $1/2$ coefficient throughout the paper.} satisfies
    \begin{equation}
        [u, v]_E + [v, u]_E = g^{-1} D_E g(u, v),
    \label{ed7}
    \end{equation}
    for all $u, v \in \mathfrak{X}(E)$, where $D_E := \rho_E^* d$.
\end{itemize}
\label{dd2}
\end{Definition}
\noindent Note that for the right-hand side of Equation (\ref{ed7}), we have
\begin{align}
    g^{-1} D_E g(f u, v) &= g^{-1} D_E g(u, f v) = g^{-1} \rho_E^* d (f g(u, v)) \nonumber\\
    &= g^{-1} \rho_E^* \left( f d g(u, v) + d f \wedge g(u, v) \right) \nonumber\\
    &= f g^{-1} \rho_E^* d g(u, v) + g^{-1} \rho_E^*(d f \wedge g(u, v)) \nonumber\\
    &= f g^{-1} D_E g(u, v) + g^{-1} \rho_E^*(d f \wedge g(u, v)),
\label{ed8}
\end{align}
for all $u, v \in \mathfrak{X}(E), f \in C^{\infty}(M, \mathbb{R})$.
\begin{Definition} \cite{14} A quadruplet  $(E, \rho_E, [\cdot,\cdot]_E, g)$ is called a metric algebroid if
\begin{itemize}
    \item $(E, \rho_E)$ is an anchored vector bundle,
    \item $[\cdot,\cdot]_E: \mathfrak{X}(E) \times \mathfrak{X}(E) \to \mathfrak{X}(E)$ is an $\mathbb{R}$-bilinear map,
    \item $g$ is an $E$-metric,
    \item The symmetric part of the bracket satisfies
    \begin{equation}
        [u, v]_E + [v, u]_E = g^{-1} D_E g(u, v), 
    \label{ed9}
    \end{equation}
    for all $u, v \in \mathfrak{X}(E)$,
    \item The ``metric invariance property" holds:
    \begin{equation}
        \rho_E(u)(g(v, w)) = g([u, v]_E, w) + g(v, [u,w]_E),
    \label{ed10}
    \end{equation}
    for all $u, v, w \in \mathfrak{X}(E)$.
\end{itemize}
\label{dd3}
\end{Definition}
\noindent Every metric algebroid satisfies the following properties \cite{14}
\begin{align}
    [u, f v]_E &= f [u, v]_E + \rho_E(u)(f) v, \nonumber\\
    [f u, v]_E &= f[u, v]_E - \rho_E(v)(f) u + g(u, v) g^{-1} D_E f,
\label{ed11}
\end{align}
for all $u, v \in \mathfrak{X}(E), f \in C^{\infty}(M, \mathbb{R})$. First one, which follows from the metric invariance property and the non-degeneracy of the $E$-metric $g$, implies that every metric algebroid is an almost-Courant algebroid. 

\begin{Definition} A metric algebroid is called a pre-Courant (resp. Courant) algebroid if it is also a pre-Leibniz (resp. Leibniz) algebroid\footnote{Note that there is a pre-Courant algebroid notion in the literature \cite{14, 15}, which is slightly more restrictive (with an additional axiom about $ker(\rho_E)$ being a $g$-coisotropic distribution) than ours.}. 
\label{dd4}
\end{Definition}

One important class of almost-Courant algebroids is the exact ones that satisfy the following exact sequence:
\begin{equation}
    0 \xrightarrow{\quad} T^*(M) \xrightarrow{\quad {g^{-1} \rho_E^*} \quad} E \xrightarrow{\ \quad \rho_E \ \quad} T(M) \xrightarrow{\quad} 0.
    \label{ed12}
\end{equation}
\v{S}evera's classical result \cite{16} states that one can classify all exact Courant algebroids by the third de Rham cohomology of the base manifold up to Courant algebroid isomorphism. Consequently, for a $g$-isotropic splitting $\phi$, the $E$-metric of an exact Courant algebroid is given by
\begin{equation} g \left( (\phi \oplus g^{-1} \rho_E^*)(U + \omega), (\phi \oplus g^{-1} \rho_E^*)(V + \eta) \right) = g_S(U + \omega, V + \eta).
\label{ed13}
\end{equation}
Moreover, the bracket can be written in terms of the standard Dorfman bracket (\ref{ed1}) and a twist by a closed 3-form $H \in \Omega^3(M)$,
\begin{equation}
    [(\phi \oplus g^{-1} \rho_E^*)(U + \omega), (\phi \oplus g^{-1} \rho_E^*)(V + \eta)]_E = (\phi \oplus g^{-1} \rho_E^*)([U + \omega, V + \eta]_D) + g^{-1} \rho_E^* \iota_U \iota_V H,
\label{ed14}
\end{equation}
for all $U, V \in \mathfrak{X}(M), \omega, \eta \in \Omega^1(M)$, 
which is usually written simply without the maps as
\begin{equation}
    [U + \omega, V + \eta]_E = [U, V] + \mathcal{L}_U \eta - \mathcal{L}_V \omega + d \iota_V \omega + \iota_U \iota_V H.
\label{ed15}
\end{equation}


\section{Higher Metric and Higher Courant Algebroids}
\label{s5}

\noindent In this section, we present another hierarchy of axioms analogously to the previous section. We start with the standard higher Courant algebroid by replacing the 1-forms with higher degree forms, which is equipped with the higher Dorfman bracket. As a full axiomatization for the standard higher Courant algebroid is not done to our knowledge, we give the definitions of higher almost-Courant, higher metric, higher pre-Courant and higher Courant algebroids, which will be our main motivations for Bourbaki algebroids.

In the previous section, we saw that one has the Dorfman bracket on the standard Courant algebroid $\mathbb{T}(M) = T(M) \oplus T^*(M)$. By replacing $T^*(M)$ with the $k$th exterior power of the cotangent bundle, $\Lambda^k(T^*(M))$, for a manifold of dimension greater than $k$, one can get the ``standard'' higher Courant algebroid $\mathbb{T}^k(M) := T(M) \oplus \Lambda^k(T^*(M))$ with the \textit{higher Dorfman bracket} $[\cdot,\cdot]_{hD}$,
\begin{equation}
    [U + \omega, V + \eta]_{hD} := [U, V] + \mathcal{L}_U \eta - \mathcal{L}_V \omega + d \iota_V \omega,
\label{ee1}
\end{equation}
for $U, V \in \mathfrak{X}(M), \omega, \eta \in \Omega^k(M)$. which is exactly of the same form as the Dorfman bracket (\ref{ed1}) with $\omega, \eta$ being a $k$-form \cite{3}. Similarly, one can introduce a $\Lambda^{k-1}(T^*(M))$-valued $\mathbb{T}^k(M)$-metric by
\begin{equation}
    g_S(U + \omega, V + \eta) := \iota_U \eta + \iota_V \omega.
\label{ee2}
\end{equation}
Consequently, one can prove
\begin{align}
    [u, f v]_{hD} &= f [u, v]_{hD} + proj_1(u)(f) v, \nonumber\\
    [u, v]_{hD} + [v, u]_{hD} &= inc_2 \ d g_S(u, v), \nonumber\\
    \mathcal{L}_{proj_1(u)}(g_S(v, w)) &= g_S([u, v]_{hD}, w) + g_S(v, [u,w]_{hD}), \nonumber\\
    proj_1([u, v]_{hD}) &= [proj_1(u), proj_1(v)], \nonumber\\
    [u, [v, w]_{hD}]_{hD} &= [[u, v]_{hD}, w]_{hD} + [v, [u, w]_{hD}]_{hD},
\label{ee3}
\end{align}
for all $u, v, w \in \mathbb{T}^k(M), f \in C^{\infty}(M, \mathbb{R})$, where $inc_2: \Lambda^k(T^*(M)) \to \mathbb{T}^k(M) = T(M) \oplus \Lambda^k(T^*(M))$ is the canonical inclusion of the second component, and $proj_1: \mathbb{T}^k(M) = T(M) \oplus \Lambda^k(T^*(M)) \to T(M)$ is the projection onto the first component. 

An axiomatization of these properties analogous to Courant algebroids has not been done to our knowledge. We believe that this is partially because of a need for an additional ingredient for higher Courant algebroids with $k > 1$, which is not transparent when one does not include the inclusion map $inc_2$ in the second property in Equation (\ref{ee3}). For the usual Courant algebroids, $inc_2$ can be identified with $g_S^{-1} \rho_E^*$. This is not the case for higher Courant algebroids with $k > 1$ because $g_S^{-1}$ would not have the appropriate domain for an $\Lambda^{k-1}(T^*(M))$-valued $\mathbb{T}^k(M)$-metric, so a map $\chi: \Lambda^k(T^*(M)) \to E$ is the additional ingredient needed.

Note that similarly each of these properties (\ref{ee3}) reduces to the ones in Equation (\ref{ed2}) for a Courant algebroid for $k = 1$. The first, fourth and fifth ones are completely identical to those of Courant algebroids, implying that the standard higher Courant algebroid is an almost-Leibniz, pre-Leibniz and Leibniz algebroid, respectively. On the other hand, the third one follows from the simple fact that $\mathcal{L}_U f = U(f)$ for a smooth function $f$, so that the left-hand side of it becomes
\begin{equation}
    \mathcal{L}_{proj_1(u)}(g_S(v, w)) = proj_1(u)(g_S(v, w)),
\label{ee4}
\end{equation}
because $g_S$ takes values in smooth functions. 

In light of these observations, we introduce the axiomatic versions of higher almost-Courant, higher metric, higher pre-Courant and higher Courant algebroids.

\begin{Definition} A quintet $(E, \rho_E, [\cdot,\cdot]_E, g, \chi)$ is called a higher almost-Courant algebroid if
\begin{itemize}
    \item $(E, \rho_E, [\cdot,\cdot]_E)$ is an almost-Leibniz algebroid,
    \item $g$ is a $\Lambda^{k-1}(T^*(M))$-valued $E$-metric,
    \item $\chi: \Lambda^k(T^*(M)) \to E$ is a vector bundle morphism,
    \item The symmetric part of the bracket satisfies
    \begin{equation}
        [u, v]_E + [v, u]_E = \chi d g(u, v), 
    \label{ee5}
    \end{equation}
    for all $u, v \in \mathfrak{X}(E)$.
\end{itemize}
\label{de0}
\end{Definition}

\begin{Definition} A quintet $(E, \rho_E, [\cdot,\cdot]_E, g, \chi)$ is called a higher metric algebroid if
\begin{itemize}
    \item $(E, \rho_E)$ is an anchored vector bundle,
    \item $[\cdot,\cdot]_E: \mathfrak{X}(E) \times \mathfrak{X}(E) \to \mathfrak{X}(E)$ is an $\mathbb{R}$-bilinear map,
    \item $g$ is a $\Lambda^{k-1}(T^*(M))$-valued $E$-metric,
    \item $\chi: \Lambda^k(T^*(M)) \to E$ is a vector bundle morphism,
    \item The symmetric part of the bracket satisfies
    \begin{equation}
        [u, v]_E + [v, u]_E = \chi d g(u, v), 
    \label{ee6}
    \end{equation}
    for all $u, v \in \mathfrak{X}(E)$,
    \item The ``higher metric invariance property" holds:
    \begin{equation}
        \mathcal{L}_{\rho_E(u)}(g(v, w)) = g([u, v]_E, w) + g(v, [u,w]_E),
    \label{ee7}
    \end{equation}
    for all $u, v, w \in \mathfrak{X}(E)$.
\end{itemize}
\label{de1}
\end{Definition}
\noindent Note that the $E$-metric $g$ is $\Lambda^{k-1}(T^*(M))$-valued whereas the domain of the vector bundle morphism $\chi$ is $\Lambda^k(T^*(M))$, and the exterior derivative is used as a map $d: \Omega^{k-1}(M) \to \Omega^k(M)$. Moreover, we have
\begin{align}
    \chi d g(f u, v) &= \chi d g (u, f v) = \chi d (f g(u, v)) ) = \chi \left( f d g(u, v) + d f \wedge g(u, v) \right) \nonumber\\
    &= f \chi d g(u, v) + \chi(d f \wedge g(u, v)),
\label{ee8}
\end{align}
for all $u, v \in \mathfrak{X}(E), f \in C^{\infty}(M, \mathbb{R})$.

\begin{Proposition} Every higher metric algebroid is a higher almost-Courant algebroid satisfying
\begin{align}
    [u, f v]_E &= f [u, v]_E + \rho_E(u)(f) v, 
\label{ee9} \\
    [f u, v]_E &= f[u, v]_E - \rho_E(v)(f) u + \chi(d f \wedge g(u, v)),
\label{ee10}
\end{align}
for all $u, v \in \mathfrak{X}(E), f \in C^{\infty}(M, \mathbb{R})$.
\label{pe1}
\end{Proposition}

\begin{proof} The proof is almost identical to the proof of Proposition 2.1 of \cite{14}. In order to prove the right-Leibniz rule (\ref{ee9}), we use the higher metric invariance property (\ref{ee7}):
\begin{equation}
    \mathcal{L}_{\rho_E(u)}(g(f v, w)) = g([u, f v]_E, w) + g(f v, [u, w]_E),
\nonumber
\end{equation}
which implies
\begin{equation}
    f \mathcal{L}_{\rho_E(u)}(g(v, w)) + \rho_E(u)(f) g(v, w) = g([u, f v]_E, w) + f g(v, [u, w]_E),
\nonumber
\end{equation}
by the $C^{\infty}(M, \mathbb{R})$-non-linearity property (\ref{eb10}) of the Lie derivative. Using the higher metric invariance property once again together with the $C^{\infty}(M, \mathbb{R})$-bilinearity of the $E$-metric $g$, we get
\begin{equation}
    g(f [u, v]_E, w) = g([u, f v]_E, w) - g(\rho_E(u)(f) v, w).
\nonumber
\end{equation}
The non-degeneracy of $g$ gives the desired result (\ref{ee8}). For the left-Leibniz rule $(\ref{ee10})$, we start with the symmetric part (\ref{ee6}) of the bracket
\begin{equation}
    [u, f v]_E + [f v, u]_E = \chi d g(u, fv),
\nonumber 
\end{equation}
which implies 
\begin{equation}
    [f v, u]_E = \chi d(f g(u, v)) - \left\{ f[u, v]_E + \rho_E(u)(f) v \right\},
\nonumber
\end{equation}
by the right-Leibniz rule (\ref{ee9}) which we proved. By using the $C^{\infty}(M, \mathbb{R})$-non-linearity property (\ref{eb10}) of the exterior derivative, we get
\begin{equation}
    [f v, u]_E = f \left\{ \chi d g(u, v) - [u, v]_E \right\} - \rho_E(u)(f) v + \chi(d f \wedge g(u, v)),
\nonumber 
\end{equation}
where we also use the $C^{\infty}(M, \mathbb{R})$-linearity of $g$ and $\chi$. By using the symmetric part (\ref{ee6}) one more time, we get the desired result due to the symmetry of the $E$-metric $g$.
\end{proof}

\begin{Definition} A higher metric algebroid is called a higher pre-Courant (resp. higher Courant) algebroid if it is also a pre-Leibniz (resp. Leibniz) algebroid.
\label{de2}
\end{Definition}

Note that, for $k = 1$, the identification $\chi = g^{-1} \rho_E^*$ and the fact that $U(f) = \mathcal{L}_U f$ indicate that every almost-Courant, metric, pre-Courant and Courant algebroid is a higher almost-Courant, higher metric, higher pre-Courant and higher Courant algebroid, respectively. Moreover, the standard higher Courant algebroid $\mathbb{T}^k(M) = T(M) \oplus \Lambda^k(T^*(M))$ equipped with the higher Dorfman bracket (\ref{ee1}) is a higher Courant algebroid due to Equation (\ref{ee3}), which is not surprising since our aim was to axiomatize these properties.

\section{Symmetric Part of the Bracket and Bourbaki Algebroids}
\label{s6}

\noindent In Sections (\ref{s4}) and (\ref{s5}), we considered the hierarchy of various algebroid structures:
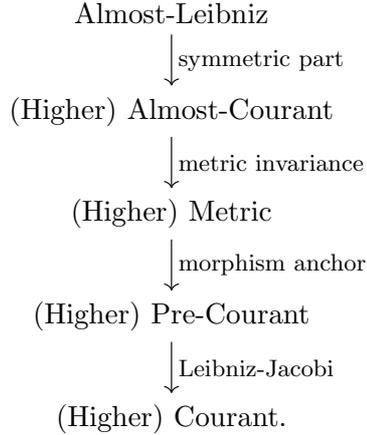
\begin{center}
\begin{tikzcd}
& \mbox{Almost-Leibniz} \arrow[d, "\mbox{\footnotesize{symmetric part}}"] \\
& \mbox{(Higher) Almost-Courant} \arrow[d, "\mbox{\footnotesize{metric invariance}}"] \\
& \mbox{(Higher) Metric} \arrow[d, "\mbox{\footnotesize{morphism anchor}}"] \\
& \mbox{(Higher) Pre-Courant} \arrow[d, "\mbox{\footnotesize{Leibniz-Jacobi}}"] \\
& \mbox{(Higher) Courant}.
\end{tikzcd}
\end{center}
\noindent These algebroids are defined with a $\Lambda^k(T^*(M))$-valued $E$-metric, which is used for two purposes. First one is that the symmetric part of the bracket is given in terms of a map that \textit{behaves} like a derivation acting on the $E$-metric as given in Equation (\ref{ed8}) and (\ref{ee8}), whereas the second one is the metric invariance property (\ref{ed10}) and (\ref{ee7}). We saw that these are enough to prove the left- and right-Leibniz rules, which are necessary to show tensoriality of objects whose constructions involve the bracket. Lastly, we considered the cases when the anchor is a morphism of brackets and the Leibniz-Jacobi identity is satisfied. In this and consecutive sections, our aim is to build up a new construction parallel to this hierarchy in a rather natural manner. Hence, we begin by considering almost-Leibniz algebroids, which satisfy the right-Leibniz rule
\begin{equation}
[u, f v]_E = f [u, v]_E + \rho_E(u)(f) v,
\label{ef1}
\end{equation}
for all $u, v \in \mathfrak{X}(E), f \in C^{\infty}(M, \mathbb{R})$.

For an almost-Lie algebroid, i.e. an almost-Leibniz algebroid whose bracket is anti-symmetric, the above right-Leibniz rule automatically fixes the left-Leibniz rule as
\begin{equation} 
[f u, v]_E = -[v, f u]_E = - \left\{ f [v, u]_E + \rho_E(v)(f) u \right\} = f [u, v]_E - \rho_E(v)(f) u .
\label{ef2}
\end{equation}
On the other hand, for the most general case when there is no anti-symmetry, the right- and left-Leibniz rules are independent from each other. Therefore, the left-Leibniz rule should be introduced separately with an additional term:
\begin{equation} 
[f u, v]_E = f [u, v]_E - \rho_E(v)(f) u + L(D f, u, v),
\label{ef3}
\end{equation}
for some maps $L$ and $D$ with appropriate domains and ranges\footnote{We will have a similar and more detailed discussion in Section (\ref{s11}) about ``$A$-symbol maps''.}. In order to have a consistent rule for the product of two smooth functions, these maps should satisfy the following conditions
\begin{itemize}
    \item $L$ should be a $C^{\infty}(M, \mathbb{R})$-multilinear map,
    \item $D$ should be a derivation on the ring of smooth functions $C^{\infty}(M, \mathbb{R})$, i.e. 
    \begin{equation} D(f_1 f_2) = f_1 D f_2 + f_2 D f_1,
    \label{ef4}
    \end{equation}
    for all $f_1, f_2 \in C^{\infty}(M, \mathbb{R})$.
\end{itemize}
The usual exterior derivative\footnote{From now on, we use the underlined versions $\underline{d}, \underline{\mathcal{L}}$ and $\underline{\iota}$ for the usual differential geometric objects; exterior derivative, Lie derivative and interior product. In consecutive sections, we will reserve the letters, $d, \mathcal{L}, \iota$, for their generalized \textit{Bourbaki} versions.} $\underline{d}$ clearly satisfy the above property, but as we are working on an anchored vector bundle there is another natural candidate given by $D_E = \rho_E^* \underline{d}$. In the literature \cite{17}, including previous papers \cite{13, 18, 19} of TD and KD, local almost-Leibniz algebroids sometimes are defined by using $D_E$, and this leads one to a $C^{\infty}(M, \mathbb{R})$-multilinear map $L: \Omega^1(E) \times \mathfrak{X}(E) \times \mathfrak{X}(E) \to \mathfrak{X}(E)$. For the purposes of the paper, the first option $\underline{d}$ is more natural, so we define local almost-Leibniz algebroids as follows:
\begin{Definition} \cite{20} A quadruplet $(E, \rho_E, [\cdot,\cdot]_E, L)$ is called a local almost-Leibniz algebroid if
\begin{itemize}
    \item $(E, \rho_E, [\cdot,\cdot]_E)$ is an almost-Leibniz algebroid,
    \item $L: \Omega^1(M) \times \mathfrak{X}(E) \times \mathfrak{X}(E) \to \mathfrak{X}(E)$ is a $C^{\infty}(M, \mathbb{R})$-multilinear map, called the locality operator, 
    \item The bracket satisfies the left-Leibniz rule
        \begin{equation} 
        [f u, v]_E = f [u, v]_E - \rho_E(v)(f) u + L(\underline{d} f, u, v),
        \label{ef5}
        \end{equation}
        for all $u, v \in \mathfrak{X}(E), f \in C^{\infty}(M, \mathbb{R})$.
\end{itemize}
\label{df1}
\end{Definition}
Note that any almost-Courant algebroid is a local almost-Leibniz algebroid with the locality operator $L(\omega, u, v) = g^{-1} \rho_E^*(\omega g(u, v))$. Moreover any higher almost-Courant algebroid is a local almost-Leibniz algebroid with the locality operator $L(\omega, u, v) = \chi(\omega \wedge g(u, v))$.

As we mentioned in the discussion around Equation (\ref{ef2}), if the symmetric part of the bracket is 0, then the left-Leibniz rule fixes the right-Leibniz rule. On the other hand, provided that the symmetric part of the bracket is given in a specific form, we can still evaluate the locality operator, as in the case of almost-Courant and higher almost-Courant algebroids\footnote{See the proof of Proposition (\ref{pe1}).}. In the next proposition, we make a small observation; this specific form is not only sufficient, but also necessary for an almost-Leibniz algebroid to be local. Hence, we get an equivalent condition for an almost-Leibniz algebroid to be local.
\begin{Proposition} Let $(E, \rho_E, [\cdot,\cdot]_E)$ be an almost-Leibniz algebroid, $L: \Omega^1(M) \times \mathfrak{X}(E) \times \mathfrak{X}(E) \to \mathfrak{X}(E)$ a $C^{\infty}(M, \mathbb{R})$-multilinear map. Then, the following are equivalent:
\begin{enumerate}
    \item $(E, \rho_E, [\cdot,\cdot]_E, L)$ is a local almost-Leibniz algebroid.
    \item The symmetric part  $S(u, v) := [u, v]_E + [v, u]_E$ of the bracket satisfies
\begin{equation}
   S(f u, v) = S(v, f u) = f S(u, v) + L(\underline{d} f, u, v),
\label{ef6}
\end{equation}
for all $u, v \in \mathfrak{X}(E), f \in C^{\infty}(M, \mathbb{R})$. 
\end{enumerate}
\label{pf1}
\end{Proposition}

\begin{proof} It can be proven by simple calculations. First, let us prove $(1 \implies 2)$:
\begin{align*}
    S(f u, v) &= [f u, v]_E + [v, f u]_E \\
    &= \left\{ f [u, v]_E - \rho_E(v)(f) u + L(\underline{d} f, u, v) \right\} + \left\{ f [v, u]_E + \rho_E(v)(f) u \right\} \\
    &= f S(u, v) + L(\underline{d} f, u, v),
\nonumber 
\end{align*}
which is the desired result. For the other direction $(2 \implies 1)$, we start with
\begin{equation*}
    S(f u, v) = f S(u, v) + L(\underline{d} f, u, v),
\end{equation*}
which yields
\begin{equation*}
    [f u, v]_E + [v, f u]_E = f \left\{ [u, v]_E + [v, u]_E \right\} + L(\underline{d} f, u, v).
\end{equation*}
By using the right-Leibniz rule, this becomes
\begin{equation*}
    [f u, v]_E = f [u, v]_E - \rho_E(v)(f) u + L(\underline{d} f, u, v),
\end{equation*}
which is the desired left-Leibniz rule.
\end{proof}

One can observe that there are two properties of the symmetric part $S$ in the above proposition: the obvious one is symmetry, and the other is having a $C^{\infty}(M, \mathbb{R})$-non-linear term in Equation (\ref{ef6}). If one tries to separate these two properties, one can suggest to write down $S$ as a composition of two maps 
\begin{equation}
S = \mathbb{D} g,
\label{ef7}
\end{equation}
where $g$ is an $R$-valued $E$-metric for some arbitrary vector bundle $R$, and $\mathbb{D}: \mathfrak{X}(R) \to \mathfrak{X}(E)$ is an $\mathbb{R}$-linear (but $C^{\infty}(M, \mathbb{R})$-non-linear) map satisfying 
\begin{equation}
\mathbb{D}(f r) = f \mathbb{D} r + \mathbb{L}_{d f} r,
\label{ef8}
\end{equation}
for all $f \in C^{\infty}(M, \mathbb{R}), r \in \mathfrak{X}(R)$, where $\mathbb{L}: \Omega^1(M) \times \mathfrak{X}(R) \to \mathfrak{X}(E)$ is a $C^{\infty}(M, \mathbb{R})$-bilinear map. One can directly expect that $g$ is a symmetric $(0, 2)$-type $E$-tensor, yet the non-degeneracy of $g$ will be crucial for many results as we will see.

In order have a better understanding of the map $\mathbb{D}$, we use\footnote{Rather, we use a slightly modified version which allows us to consider maps with ranges and domains that can be different.} differential operator bundles \cite{9}, \cite{21}. Let us consider two vector bundles $R$ and $Z$. A zeroth-order differential operator from $R$ to $Z$ is a $C^{\infty}(M, \mathbb{R})$-linear map $\mathfrak{X}(R) \to \mathfrak{X}(Z)$, so it is just a vector bundle morphism, whose set will we denoted by $Hom(R, Z)$. On the other hand, a first-order differential operator is an $\mathbb{R}$-linear map $D: \mathfrak{X}(R) \to \mathfrak{X}(Z)$ such that for all $f \in C^{\infty}(M, \mathbb{R})$, the map
\begin{equation} r \mapsto D(f r) - f D(r)
\label{ef9}
\end{equation}
is a zeroth-order differential operator\footnote{Local almost-Leibniz algebroids can also be defined as algebroids whose brackets are first-order differential operators, see for example \cite{22}.}. Every zeroth-order differential operator is also first-order, so we have a natural inclusion $Hom(R, Z) \hookrightarrow Dif^1(R, Z)$, where $Dif^1(R, Z)$ denotes the set of first-order differential operators from $R$ to $Z$. Moreover, we can define a vector bundle morphism 
\begin{align} 
\sigma: Dif^1(R, Z) &\to Hom(T^*(M), Hom(R, Z)), \nonumber\\
D &\mapsto \left( \underline{d} f \mapsto \left( r \mapsto D(f r) - f D(r) \right) \right),
\label{ef10}
\end{align}
which is called the \textit{symbol map} of $D$. In this context, the above map $\mathbb{D}$ should be a first-order differential operator from $R$ to $E$, whose symbol is given by $\mathbb{L}$, where this identification follows from tensor-hom adjunction. This leads us to the following definition of almost-Bourbaki algebroids, which will be at the level of almost-Courant algebroids in the hierarchy that we are trying to construct.
\begin{Definition} A septet $(E, \rho_E, [\cdot,\cdot]_E, R, g, \mathbb{D}, \mathbb{L})$ is called an almost-Bourbaki algebroid if 
\begin{itemize}
\item $(E, \rho_E, [\cdot,\cdot]_E)$ is an almost-Leibniz algebroid, 
\item $R$ is a vector bundle, 
\item $g$ is an $R$-valued $E$-metric, 
\item $\mathbb{D}: \mathfrak{X}(R) \to \mathfrak{X}(E)$ is a first-order differential operator, whose symbol is given by $\mathbb{L}$,
\item The symmetric part of the bracket satisfies
\begin{equation}
[u, v]_E + [v, u]_E = \mathbb{D} g(u, v) ,
\label{ef11}
\end{equation}
for all $u, v \in \mathfrak{X}(E)$.
\end{itemize}
\label{df3}
\end{Definition}

\begin{Definition}
An almost-Bourbaki algebroid is called a pre-Bourbaki (resp. Bourbaki) algebroid if it is also a pre-Leibniz (resp. Leibniz) algebroid.
\label{df4}
\end{Definition}
With this new terminology, Proposition (\ref{pf1}) implies that every almost-Bourbaki algebroid is a local almost-Leibniz algebroid; a simple yet useful result for our constructions.
\begin{Corollary} If $(E, \rho_E, [\cdot,\cdot]_E, R, g, \mathbb{L}, \mathbb{D})$ is an almost-Bourbaki algebroid, then $(E, \rho_E, [\cdot,\cdot]_E, L)$ is a local almost-Leibniz algebroid, where the locality operator $L$ is given by
\begin{equation}
L(\omega, u, v) = \mathbb{L}_{\omega}(g(u, v)).
\label{ef12}
\end{equation}
\label{cf1}
\end{Corollary}

\begin{proof} The symmetric part of the bracket is given by $S(u, v) = \mathbb{D} g(u, v)$ which satisfies 
\begin{equation}
    S(f u, v) = \mathbb{D} g(f u, v) = \mathbb{D}(f g(u, v)) = f \mathbb{D} g(u, v) + \mathbb{L}_{\underline{d} f}(g (u, v)),
    \nonumber
\end{equation}
so that one can identify $L(\underline{d} f, u, v)$ with $\mathbb{L}_{\underline{d} f}(g (u, v))$, and extend it to all 1-forms as desired.
\end{proof}
\noindent Note that this corollary forces one to have a locality operator that is symmetric in $u$ and $v$ for an almost-Bourbaki algebroid.


\section{Metric Invariance Property and Metric-Bourbaki Algebroids}
\label{s7}

\noindent For metric and higher metric algebroids, we saw that the right-Leibniz rule is implied by the metric invariance property (\ref{ed10}) or (\ref{ee7}), where these properties involve a Lie derivative acting on $(k-1)$-forms, which is a map of the form $\underline{\mathcal{L}}: \mathfrak{X}(M) \times \Omega^{k-1}(M) \to \Omega^{k-1}(M)$. When we consider an $E$-metric $g$ taking values in an arbitrary vector bundle $R$, there is no canonically defined operators analogous to Lie derivative. In our prototypical example higher almost-Courant algebroids, the $E$-metric is $\Lambda^{k-1}(T^*(M))$-valued, so that in this case we identify $R$ with $\Lambda^{k-1}(T^*(M))$. Extending this analogy, we need to introduce a map $\mathcal{L}^R: \mathfrak{X}(M) \times \mathfrak{X}(R) \to \mathfrak{X}(R)$ defined by\footnote{More generally, $g([u, v]_E, w) + g(v, [u, w]_E)$ should be a map of the form $\mathcal{L}^R(u, v, w)$. Due to symmetry in $v$ and $w$, we can decompose it as $\mathcal{L}^R_u (g(v, w))$ similarly to what we have done for the operator $\mathbb{D}$. This decomposition seems necessary to interpret it as a metric invariance property. Yet, we need to make another, seemingly \textit{ad hoc}, assumption that it is given in the form $\mathcal{L}^R_{\rho_E(u)}(g(v, w))$, which is not problematic when $E$ is transitive. We plan to investigate the details of the case without this assumption in the near future. Moreover, we do not care whether $g$ is surjective on $R$ or not; we will consider maps $\mathcal{L}^R_{\rho_E(u)}$ with the whole domain $R$.} 
\begin{equation}
    \mathcal{L}^R_{\rho_E(u)}(g(v, w)) := g([u, v]_E, w) + g(v, [u, w]_E).
\label{eg1}
\end{equation}
We are interested to find a sufficient condition on the map $\mathcal{L}^R_{\rho_E(u)}$ so that the right-Leibniz rule holds for the bracket $[\cdot,\cdot]_E$. In order to find such a condition, we assume the right-Leibniz identity for the bracket, and consider
\begin{align}
    \mathcal{L}^R_{\rho_E(u)} (g(f v, w)) &= g([u, f v]_E, w) + g(f v, [u, w]_E) \nonumber\\
    &= g(f [u, v]_E + \rho_E(u)(f) v, w) + f g(v, [u, w]_E) \nonumber\\
    &= f \mathcal{L}^R_{\rho_E(u)} (g(v, w)) + \rho_E(u)(f) g(v, w),
\label{eg2}
\end{align}
for all $u, v, w \in \mathfrak{X}(E), f \in C^{\infty}(M, \mathbb{R})$. Hence, if $\mathcal{L}^R_{\rho_E(u)}$ is a first-order differential operator from $R$ to itself with the symbol given by the anchor $\rho_E$, then we see that the metric invariance property (\ref{eg1}) implies the right-Leibniz rule by the non-degeneracy of $g$. As the symmetric part and the right-Leibniz rule imply the left-Leibniz rule, we see that this metric invariance property and the symmetric part are sufficient to have a local almost-Leibniz algebroid. 

The map $\mathcal{L}^R_{\rho_E(u)}$ is a rather special kind of first-order differential operator, namely it is a derivation \cite{9}. A first-order differential operator $D: \mathfrak{X}(R) \to \mathfrak{X}(R)$ is called a \textit{derivation} on $R$ if the symbol of $D$ is given in  terms of a vector field, i.e. for some usual vector field $V_D$,
\begin{equation} D(f r) = f D r + V_D(f) r,
\label{eg3}
\end{equation}
for all $f \in C^{\infty}(M, \mathbb{R}), r \in \mathfrak{X}(R)$. The set of all derivations on $R$ is denoted by $\mathfrak{D}(R)$, and it fits into the following short exact sequence
\begin{equation}
    0 \xrightarrow{\quad} Hom(R, R) \xrightarrow{\quad inc \quad} \mathfrak{D}(R) \xrightarrow{\ \quad \rho_{\mathfrak{D}(R)} \ \quad} T(M) \xrightarrow{\quad} 0.
\label{eg4}
\end{equation}
In this sense, $\mathfrak{D}(R)$ can be seen as an anchored vector bundle with $\rho_{\mathfrak{D}(R)}(D) = V_D$\footnote{Moreover, the usual commutator of operators yields a transitive Lie algebroid structure on $\mathfrak{D}(R)$, called the (Atiyah) gauge Lie algebroid or covariant operator bundle \cite{9}.}. By Equation (\ref{eg2}), we have 
\begin{equation} \rho_{\mathfrak{D}(R)} \left( \mathcal{L}^R_{\rho_E(u)} \right) = \rho_E(u),
\label{eg5}
\end{equation}
for all $u \in \mathfrak{X}(E)$. Note that, this is valid for the usual Lie derivative $\underline{\mathcal{L}}$ when we identify $R$ with $\Lambda^{k-1}(T^*(M))$, since $\underline{\mathcal{L}}_V (f \omega) = f \underline{\mathcal{L}}_V \omega + V(f) \omega$ by Equation (\ref{eb10}), so that we have
\begin{equation}
\rho_{\mathfrak{D}(\Lambda^{k-1}(T^*(M)))} \left( \underline{\mathcal{L}}_{\rho_E(u)} \right) = \rho_E(u)(f).
\label{eg5b}
\end{equation} This discussion leads us to the definition of almost-metric-Bourbaki algebroids, which will be at the level of metric algebroids in the hierarchy that we are trying to construct:
\begin{Definition} An octet $(E, \rho_E, [\cdot,\cdot]_E, R, g, \mathbb{D}, \mathbb{L}, \mathcal{L}^R)$ is called an almost-metric-Bourbaki algebroid if
\begin{itemize}
    \item $(E, \rho_E)$ is an anchored vector bundle,
    \item $[\cdot,\cdot]_E: \mathfrak{X}(E) \times \mathfrak{X}(E) \to \mathfrak{X}(E)$ is an $\mathbb{R}$-bilinear map,
    \item $R$ is a vector bundle,
    \item $g$ is an $R$-valued $E$-metric,
    \item $\mathbb{D}: \mathfrak{X}(R) \to \mathfrak{X}(E)$ is a first-order differential operator, whose symbol is given by $\mathbb{L}$,
    \item The symmetric part of the bracket satisfies
    \begin{equation}
        [u, v]_E + [v, u]_E = \mathbb{D} g(u, v), 
    \label{eg6}
    \end{equation}
    for all $u, v \in \mathfrak{X}(E)$,
    \item $\mathcal{L}^R: T(M) \to \mathfrak{D}(R)$ is an $\mathbb{R}$-linear map\footnote{This map is not $C^{\infty}(M, \mathbb{R})$-linear, in general.} such that $\rho_{\mathfrak{D}(R)} \mathcal{L}^R = \rho_E$,
    \item The $R$-valued $E$-metric $g$ is invariant in the sense that
    \begin{equation}
        \mathcal{L}^R_{\rho_E(u)}(g(v, w)) = g([u, v]_E, w) + g(v, [u,w]_E),
    \label{eg7}
    \end{equation}
    for all $u, v, w \in \mathfrak{X}(E)$.
\end{itemize}
\label{dg1}
\end{Definition}
\noindent In this new terminology, the discussion around Equation (\ref{eg2}) implies that every almost-metric-Bourbaki algebroid is an almost-Bourbaki algebroid, and consequently a local almost-Leibniz algebroid, where this result is completely analogous to Proposition (\ref{pe1}). Moreover, as the left-Leibniz rule is also satisfied, we get
\begin{align} 
\mathcal{L}^R_{f \rho_E(u)}(g(v, w)) &= \mathcal{L}^R_{\rho_E(f u)}(g(v, w)) = g([f u, v]_E, w) + g(v, [f u, w]_E) \nonumber\\
&= f \left\{ g([u, v]_E, w) + g(v, [u, w]_E) \right\} - \rho_E(v)(f) g(u, w) - \rho_E(w)(f) g(v, u) \nonumber\\
& \quad + g \left( \mathbb{L}_{\underline{d} f} g(u, v), w \right) + g \left( v, \mathbb{L}_{\underline{d} f} g(u, w) \right), \nonumber\\
&= f \mathcal{L}^R_{\rho_E(u)}(g(v, w)) - \rho_E(v)(f) g(u, w) - \rho_E(w)(f) g(v, u) \nonumber\\
& \quad + g \left( \mathbb{L}_{\underline{d} f} g(u, v), w \right) + g \left( v, \mathbb{L}_{\underline{d} f} g(u, w) \right),
\label{eg8}
\end{align}
for all $u, v, w \in \mathfrak{X}(E), f \in C^{\infty}(M, \mathbb{R})$. Hence, $\mathcal{L}^R$ is also a first-order differential operator in the first entry, with a rather complicated symbol.

\begin{Definition} An almost-metric-Bourbaki algebroid is called a pre-metric-Bourbaki (resp. metric-Bourbaki) algebroid if it is also a pre-Leibniz (resp. Leibniz) algebroid.
\label{dg2}
\end{Definition}
With these definitions, we established the following hierarchy of algebroid structures that we are trying to construct:
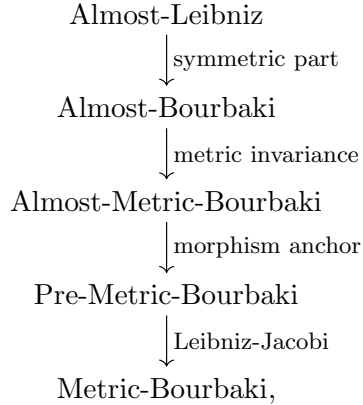
\begin{center}
\begin{tikzcd}
& \mbox{Almost-Leibniz} \arrow[d, "\mbox{\footnotesize{symmetric part}}"] \\
& \mbox{Almost-Bourbaki} \arrow[d, "\mbox{\footnotesize{metric invariance}}"] \\
& \mbox{Almost-Metric-Bourbaki} \arrow[d, "\mbox{\footnotesize{morphism anchor}}"] \\
& \mbox{Pre-Metric-Bourbaki} \arrow[d, "\mbox{\footnotesize{Leibniz-Jacobi}}"] \\
& \mbox{Metric-Bourbaki},
\end{tikzcd}
\end{center}
\noindent which is completely analogous to Courant and higher algebroid hierarchy that we consider in Sections (\ref{s4}) and (\ref{s5}).





\section{Exact Bourbaki Algebroids}
\label{s8}

\noindent In Section (\ref{s4}), we considered an important class of almost-Courant algebroids that are called exact. These exact ones fit into the following exact sequence of vector bundles:
\begin{equation}
    0 \xrightarrow{\quad} T^*(M) \xrightarrow{\quad {g^{-1} \rho_E^*} \quad} E \xrightarrow{\ \quad \rho_E \ \quad} T(M) \xrightarrow{\quad} 0.
    \label{eh1}
\end{equation}
Moreover, for any almost-Courant algebroid, the first-order differential operator $\mathbb{D}$ and its symbol $\mathbb{L}$ decompose into two parts
\begin{equation}
    \mathbb{D} = (g^{-1} \rho_E^*) \underline{d}, \qquad \qquad \qquad \mathbb{L}_{\omega} f = (g^{-1} \rho_E^*)(\omega f),
\label{eh2}
\end{equation}
 where the first part $g^{-1} \rho_E^*$ enters the exact sequence above. On the other hand, for a higher almost-Courant algebroid $(E, \rho_E, [\cdot,\cdot]_E, g, \chi)$, the maps $\mathbb{D}$ and $\mathbb{L}$ decompose into two as
 \begin{equation}
    \mathbb{D} = \chi \underline{d}, \qquad \qquad \qquad
    \mathbb{L}_{\omega}\eta = \chi(\omega \wedge \eta),
\label{eh3}
\end{equation}
and the natural exact sequence would be
\begin{equation}
    0 \xrightarrow{\quad} \Lambda^k(T^*(M)) \xrightarrow{\quad \chi \quad} E \xrightarrow{\ \quad \rho_E \ \quad} T(M) \xrightarrow{\quad} 0,
    \label{eh4}
\end{equation}
which has the same map $\chi$. With this observation, we are led to define the exact versions of almost-Bourbaki and almost-metric-Bourbaki algebroids.
\begin{Definition} A nonet $(E, \rho_E, [\cdot,\cdot]_E, R, g, Z, \chi, d, l)$ is called an exact almost-Bourbaki algebroid if
\begin{itemize}
    \item $Z$ is a vector bundle,
    \item $\chi: Z \to E$ is a vector bundle morphism,
    \item $(E, \rho_E, [\cdot,\cdot]_E, R, g, \mathbb{D} = \chi d, \mathbb{L} = \chi l)$ is an almost-Bourbaki algebroid,
    \item The following is an exact sequence of vector bundles
    \begin{equation}
        0 \xrightarrow{\quad} Z \xrightarrow{\quad \chi \quad} E \xrightarrow{\ \quad \rho_E \ \quad} T(M) \xrightarrow{\quad} 0.
    \label{eh5}
    \end{equation}
\end{itemize}
A dectet $(E, \rho_E, [\cdot,\cdot]_E, R, g, Z, \chi, d, l, \mathcal{L}^R)$ is called an exact almost-metric-Bourbaki algebroid if $Z$ and $\chi$ are exactly as above, and $(E, \rho_E, [\cdot,\cdot]_E, R, g, \mathbb{D} = \chi d, \mathbb{L} = \chi l, \mathcal{L}^R)$ is an almost-metric algebroid.
\label{dh1}
\end{Definition}
\noindent By construction, $d$ is a first-order differential operator from $R$ to $Z$ whose symbol is given by $l$. Moreover, as we consider real and finite-rank vector bundles over a paracompact manifold, like any short exact sequence of vector bundles, the exact sequence (\ref{eh5}) also splits. Hence, we get a vector bundle morphism $\phi: T(M) \to E$ such that $\rho_E \phi = id_{T(M)}$. Furthermore, we get an isomorphism 
\begin{equation}
    \phi \oplus \chi: T(M) \oplus Z \to E.
\label{eh6}
\end{equation}
In particular, this implies that exact higher almost-Courant algebroids are isomorphic to $\mathbb{T}^k(M) = T(M) \oplus \Lambda^k(T^*(M))$ as vector bundles.


\section{Bourbaki Pre-Calculus}
\label{s9}

\noindent In this section, we analyze the relation between the higher Dorfman bracket and Cartan calculus with the insight from \v{S}evera classification. We use this analysis to construct a new calculus, or rather a pre-calculus\footnote{We lied; we will never define a calculus in this paper, but only a pre-calculus. In a consecutive paper, we are planning to deal with the Leibniz-Jacobi identity, which yield certain extra conditions on a pre-calculus, making it into a ``calculus''.}, generalizing certain relations between the usual exterior derivative, Lie derivative and interior product. We also investigate the relation between Bourbaki algebroids and this new \textit{Bourbaki pre-calculus}, and present our main results.

In Section (\ref{s7}), we discussed that for a transitive almost-Bourbaki algebroid, we have a map
\begin{equation} \mathcal{L}^R: \mathfrak{X}(M) \times \mathfrak{X}(R) \to \mathfrak{X}(R),
\label{ei1}
\end{equation}
which satisfies
\begin{equation}
    \mathcal{L}^R_V (f r) = f \mathcal{L}^R_V r + V(f) r,
\label{ei2}
\end{equation}
for all $V \in \mathfrak{X}(M), f \in C^{\infty}(M, \mathbb{R}), r \in \mathfrak{X}(R)$. Moreover, in Section (\ref{s8}) for exact almost-Bourbaki algebroids\footnote{Note that every exact almost-Bourbaki algebroid is transitive by definition.}, we considered the map
\begin{equation} d: \mathfrak{X}(R) \to \mathfrak{X}(Z),
\label{ei3}
\end{equation}
which satisfies
\begin{equation} d (f r) = f d r + l_{\underline{d} f} r,
\label{ei4}
\end{equation}
for all $f \in C^{\infty}(M, \mathbb{R}), r \in \mathfrak{X}(R)$. Equation (\ref{ei2}) and (\ref{ei4}) are analogous to Equation (\ref{eb10}) of the usual Lie derivative $\underline{\mathcal{L}}$ and exterior derivative $\underline{d}$, which coincide with the above maps $\mathcal{L}^R$ and $d$ for a higher almost-Courant algebroid $\mathbb{T}^k(M) = T(M) \oplus \Lambda^k(T^*(M))$:
\begin{align} 
    \underline{\mathcal{L}}&: \mathfrak{X}(M) \times \mathfrak{X}(R) = \mathfrak{X}(M) \times \Omega^{k-1}(M) \to \mathfrak{X}(R) = \Omega^{k-1}(M), \nonumber\\
    \underline{d}&: \mathfrak{X}(R) = \Omega^{k-1}(M) \to \mathfrak{X}(Z) = \Omega^k(M).
\label{ei5}
\end{align}
Note that $\mathcal{L}^R$ generalizes the usual Lie derivative acting on ($k-1$)-forms. In order to construct a generalization acting on $Z$, which would be analogous to the Lie derivative of $k$-forms, we consider the following expression\footnote{This expression is not surprising when one considers the proof of \v{S}evera classification. See \cite{22} for a somewhat similar discussion.} for a right-splitting $\phi$ of an exact almost-Bourbaki algebroid
\begin{equation} 
    [\phi(V), \chi(z)]_E,
\label{ei6}
\end{equation}
which is a map of the form $\mathfrak{X}(M) \times \mathfrak{X}(Z) \to \mathfrak{X}(E)$. In order to get a map of the form $\mathcal{L}^Z: \mathfrak{X}(M) \times \mathfrak{X}(Z) \to \mathfrak{X}(Z)$ from this expression, we need to show that for all $z \in \mathfrak{X}(Z), V \in \mathfrak{X}(M)$; $[\phi(V), \chi(z)]_E$ is in the image of $\chi$. The following simple, yet useful, lemma guarantees this result for exact pre-Bourbaki algebroids. Consequently, for exact pre-Bourbaki algebroids, we get a map $\mathcal{L}^Z: \mathfrak{X}(M) \times \mathfrak{X}(Z) \to \mathfrak{X}(Z)$ such that 
\begin{equation} \chi \mathcal{L}^Z_V z = [\phi(V), \chi(z)]_E,
\label{ei6a}
\end{equation}
for all $V \in \mathfrak{X}(M), z \in \mathfrak{X}(Z)$. Moreover, this map satisfies
\begin{align}
    \mathcal{L}^Z_{f U} z &= f \mathcal{L}^Z_U z + l_{\underline{d} f} g(\phi(U), \chi(z)), \nonumber\\
    \mathcal{L}^Z_U (f z) &= f \mathcal{L}^Z_U z + U(f) z,
\label{ei6b}
\end{align}
for all $U \in \mathfrak{X}(M), f \in C^{\infty}(M, \mathbb{R}), z \in \mathfrak{X}(Z)$, which are completely analogous to Equation (\ref{eb10}) for the usual Lie derivative. These mean that $\mathcal{L}^Z$ is a first-order differential operator for both entries.
\begin{Lemma} Let $(E, \rho_E, [\cdot,\cdot]_E, R, g, Z, \chi, d, l)$ be an exact pre-Bourbaki algebroid. Then, $[u, \chi(z)]_E$ and $[\chi(z), u)]_E$ are in the image of $\chi$ for all $z \in \mathfrak{X}(Z), u \in \mathfrak{X}(E)$.
\label{li1}
\end{Lemma}

\begin{proof} By exactness, we need to show that $[\chi(z), u]_E$ is in the kernel of the anchor $\rho_E$. As the anchor is a morphism of brackets for a pre-Bourbaki algebroid, we have
\begin{equation}
    \rho_E([u, \chi(z)]_E) = [\rho_E(u), \rho_E \chi(z)] = [\rho_E(u), 0] = 0, \nonumber
\end{equation}
which is the desired result. The other case follows identically.
\end{proof}
After finding out a suitable generalization for the Lie derivative acting on $k$-forms, the only missing ingredient is the generalization of the usual interior product $\underline{\iota}$ which is a map of the form $\mathfrak{X}(M) \times \Omega^k(M) \to \Omega^{k-1}(M)$. As $\underline{\iota}$ is $C^{\infty}(M, \mathbb{R})$-bilinear, one would seek such a generalization and the following expression\footnote{This is automatically valid for exact pre-Courant algebroids and the usual interior product, and it is widely used in the literature in different contexts as we will see in our examples.} does the trick:
\begin{equation} \iota_V z := g(\chi(z), \phi(V)).
\label{ei7}
\end{equation}
As the $E$-metric $g$ is $R$-valued, the domain of $\iota$ is $R$ as desired, and we get a $C^{\infty}(M, \mathbb{R})$-bilinear map $\iota: \mathfrak{X}(M) \times \mathfrak{X}(Z) \to \mathfrak{X}(R)$. With this new notation, the first equation in (\ref{ei6b}) becomes 
\begin{equation} \mathcal{L}^Z_{f U} z = f \mathcal{L}^Z_U z + l_{\underline{d} f} \iota_U z.
\label{ei7b}
\end{equation}


\begin{Proposition} For an exact pre-metric-Bourbaki algebroid $(E, \rho_E, [\cdot,\cdot]_E, R, g, Z, \chi, d, l, \mathcal{L}^R)$, one has
\begin{equation}
    \iota_U \left( \mathcal{L}^Z_V z - d \iota_V z \right) = - \iota_V \left( \mathcal{L}^Z_U z - d \iota_U z \right).
\label{ei8}
\end{equation}
Furthermore if the map $\chi$ is $g$-isotropic, then the following also holds
\begin{equation}
    \mathcal{L}^R_U \iota_V z = \iota_{[U, V]} z + \iota_V \mathcal{L}^Z_U z,
\label{ei9}
\end{equation}
for all $U, V \in C^{\infty}(M, \mathbb{R}), z \in \mathfrak{X}(Z)$. 
\label{pi1}
\end{Proposition}

\begin{proof}
In order to prove Equation (\ref{ei8}), we consider the metric invariance property (\ref{eg8}) with $u = \chi(z), v = \phi(U), w = \phi(V)$ for some $z \in \mathfrak{X}(Z), U, V \in \mathfrak{X}(M)$, where $\phi$ is a right-splitting. By exactness and the $\mathbb{R}$-linearity of $\mathcal{L}^R$, the left-hand side of the metric invariance property vanishes. By the symmetric part of the bracket (\ref{eg7}), the right-hand side of the metric invariance property becomes 
\begin{align}
    RHS &= g([\chi(z), \phi(U)]_E, \phi(V)) + g(\phi(U), [\chi(z), \phi(V)]_E) \nonumber\\
    &= g(- [\phi(U), \chi(z)]_E + \chi d g(\phi(U), \chi(z)), \phi(V)) + U \leftrightarrow V \nonumber\\
    &= g(- \chi(\mathcal{L}^Z_U z), \phi(V)) + \iota_V d \iota_U z + U \leftrightarrow V \nonumber\\
    &= \left( - \iota_V \mathcal{L}^Z_U z + \iota_V d \iota_U z \right) + \left( - \iota_U \mathcal{L}^Z_V z + \iota_U d \iota_V z \right), \nonumber
\end{align}
which yields the desired result as the left-hand side is zero. Here, we use the definition (\ref{ei7}) of $\iota$, exactness, Equation (\ref{ei6a}) for $\mathcal{L}^Z$ and the fact that the anchor is a morphism of brackets. In order to prove Equation (\ref{ei9}), we again consider the metric invariance property (\ref{eg8}), but now we choose $u = \phi(U), v = \phi(V), w = \chi(z)$ for some $z \in \mathfrak{X}(Z), U, V \in \mathfrak{X}(M)$. The left-hand side of the metric invariance property then becomes $\mathcal{L}^R_U \iota_V z$. Moreover, the right-hand side becomes
\begin{align}
    RHS &= g([\phi(U), \phi(V)]_E, \chi(z)) + g(\phi(V), [\phi(U), \chi(z)]_E) \nonumber\\
    &= \iota_{\rho_E([\phi(U), \phi(V)]_E)} z + g(\phi(V), \chi(\mathcal{L}^Z_U z)) \nonumber\\
    &= \iota_{[U, V]} z + \iota_V \mathcal{L}^Z_U z, \nonumber
\end{align}
since $g$-isotropy of $\chi$ implies 
\begin{equation}
    g(\chi(z), u) = \iota_{\rho_E(u)} z, \nonumber
\end{equation}
for all $z \in \mathfrak{X}(Z), u \in \mathfrak{X}(E)$, which implies the desired result.
\end{proof}
When we consider an exact pre-metric-Bourbaki algebroid with a $g$-isotropic $\chi$, choosing $u = \phi(U), v = \phi(V), w = \chi(z)$ in Equation (\ref{eg8}) implies
\begin{equation}
    \mathcal{L}^R_{f U} \iota_V z = f \mathcal{L}^R_U \iota_V z - V(f) \iota_U z + \iota_V l_{\underline{d} f} \iota_U z,
\label{ei9c}
\end{equation}
for $U, V \in \mathfrak{X}(M), f \in C^{\infty}(M, \mathbb{R}), z \in \mathfrak{X}(Z)$.

In this section, we focused on many properties of the maps $\iota, d, l, \mathcal{L}^R$ and $\mathcal{L}^Z$. Now, we axiomatize a certain subset of these properties, and define a generalization of Cartan calculus. We will see that the chosen properties will be crucial for our main theorems with the insight from \v{S}evera classification.
\begin{Definition} Let $(R, Z)$ be a pair of vector bundles. The quintet $(\iota, d, l, \mathcal{L}^R, \mathcal{L}^Z)$ is called a Bourbaki pre-calculus if 
\begin{itemize}
\item $\iota: \mathfrak{X}(M) \times \mathfrak{X}(Z) \to \mathfrak{X}(R)$ is a $C^{\infty}(M, \mathbb{R})$-bilinear map,
\item $d: \mathfrak{X}(R) \to \mathfrak{X}(Z)$ is a first-order differential operator with the symbol $l$,
\item $\mathcal{L}^R: \mathfrak{X}(M) \times \mathfrak{X}(R) \to \mathfrak{X}(R)$ is an $\mathbb{R}$-bilinear map,
\item $\mathcal{L}^Z: \mathfrak{X}(M) \times \mathfrak{X}(Z) \to \mathfrak{X}(Z)$ is a first-order differential operator satisfying
    \begin{enumerate}
        \item $\mathcal{L}^Z_V(f z) = f \mathcal{L}^Z_V z + V(f) z$,
        \item $\mathcal{L}^Z_{f V} z = f \mathcal{L}^Z_V z + l_{\underline{d} f} \iota_V z$,
    \end{enumerate}
\item They satisfy
\begin{align}
    &\mathcal{L}^R_U \iota_V z = \iota_{[U, V]} z + \iota_V \mathcal{L}^Z_U z, \label{BC1} \\
    &\iota_U \left( \mathcal{L}^Z_V z - d \iota_V z \right) = - \iota_V \left( \mathcal{L}^Z_U z - d \iota_U z \right) \label{BC2}
\end{align}
\end{itemize}
for all $f \in C^{\infty}(M, \mathbb{R}), U, V \in \mathfrak{X}(M), z \in \mathfrak{X}(Z)$.
\label{di1}
\end{Definition}

\begin{Definition} Let $(E, \rho_E, [\cdot,\cdot]_E, R, g, Z, \chi, d, l, \mathcal{L}^R)$ be an exact pre-metric-Bourbaki algebroid such that $\chi$ is $g$-isotropic. The Bourbaki pre-calculus $(\iota, d, l, \mathcal{L}^R, \mathcal{L}^Z)$, where $\iota$ and $\mathcal{L}^Z$ are defined by Equation (\ref{ei7}) and Equation (\ref{ei6}), is called the ``induced'' Bourbaki pre-calculus by $E$.
\label{di2}
\end{Definition}

The usual Cartan calculus $(\underline{\iota}, \underline{d}, \underline{\mathcal{L}})$ is a Bourbaki pre-calculus $(\iota = \underline{\iota}, d = \underline{d}, l_{\omega} \eta = \omega \wedge \eta, \mathcal{L}^R = \underline{\mathcal{L}}, \mathcal{L}^Z = \underline{\mathcal{L}})$ for $\omega \in \Omega^1(M), \eta \in \Omega^k(M)$. In this sense, the maps $\iota, d$ and $\mathcal{L}^R, \mathcal{L}^Z$ generalize the notion of interior product, exterior derivative and Lie derivative, respectively. Moreover, the map $l$ generalizes the wedge product with a $1$-form $\omega$. The first three properties are straightforward generalizations of Equations (\ref{eb10}), and the forth one is identical to Equation (\ref{eb7}). On the other hand, the last one is a combination of Cartan magic formula (\ref{eb6}) and the fact that the usual interior product squares to zero (\ref{eb5}).

Having all the necessary maps, we are inclined to construct a Dorfman-like bracket, and the next theorem shows that this bracket satisfies some nice properties completely analogous to Dorfman and higher Dorfman brackets. In particular, we construct an exact pre-metric-Bourbaki algebroid from a given Bourbaki pre-calculus data.
\begin{Theorem} Let $(\iota, d, l, \mathcal{L}^R, \mathcal{L}^Z)$ be a Bourbaki pre-calculus on the pair of vector bundles $(R, Z)$, then $T(M) \oplus Z$ becomes an exact pre-metric-Bourbaki algebroid
\begin{equation}
    (E = T(M) \oplus Z, \rho_E = proj_1, [\cdot,\cdot]_S, R, g_S, Z, \chi = inc_2, d, l, \mathcal{L}^R),
\label{ei10}
\end{equation}
where the ``standard bracket'' $[\cdot,\cdot]_S$ is given by
\begin{equation} 
    [U + z, V + y]_S := [U, V] + \mathcal{L}^Z_U y - \mathcal{L}^Z_V z + d \iota_V z,
\label{ei11}
\end{equation}
and the standard $R$-valued $E$-metric $g_S$ is defined as 
\begin{equation}
    g_S(U + z, V + y) := \iota_U y + \iota_V z.
\label{ei12}
\end{equation}
Here, $proj_1: E = T(M) \oplus Z \to T(M)$ is the projection onto the first component, and $inc_2: Z \to E = T(M) \oplus Z$ is the canonical inclusion. 
\label{ti1}
\end{Theorem}

\begin{proof}
We need to check three properties:
\begin{enumerate}
    \item The symmetric part is given by Equation (\ref{eg6}),
    \item The metric invariance property (\ref{eg7}) holds,
    \item Exactness, which is straightforward.
\end{enumerate}
For the symmetric part, we clearly have
\begin{equation} [U + z, V + y]_S + [V + y, U + z]_S = d \iota_V z + d \iota_U y = d g_S(U + z, V + y). \nonumber
\end{equation}
For the metric invariance property, by the assumption (\ref{BC1}), we have
\begin{equation} \mathcal{L}^R_{\rho_E(U + z)} g_S(V + y, W + x) = \iota_{[U, V]} x + \iota_V \mathcal{L}^Z_U x + \iota_{[U, W]} y + \iota_W \mathcal{L}^Z_U y, \nonumber
\end{equation}
and by the definition (\ref{ei11}) of the standard bracket, we get
\begin{equation} g_S([U + z, V + y]_S, W + x) = \iota_{[U, V]} x + \iota_W \left( \mathcal{L}^Z_U y - \mathcal{L}^Z_V z + d \iota_V z \right). \nonumber
\end{equation}
These two equations together with the assumption (\ref{BC2}) imply the metric invariance property. 
\end{proof}

\noindent The next natural question is that whether every exact pre-metric-Bourbaki algebroid can be constructed from the standard one via a ``twist'' in the bracket. The answer is affirmative if the map $\chi$ is $g$-isotropic.
\begin{Theorem}
Let $(\iota, d, l, \mathcal{L}^R, \mathcal{L}^Z)$ be a Bourbaki pre-calculus\footnote{Note that we do not assume that $\iota$ and $\mathcal{L}^Z$ are the induced ones. We are going the prove that this is necessarily the case though.}, and $(E, \rho_E, [\cdot,\cdot]_E, R, g, Z, \chi, d, l, \mathcal{L}^R)$ an exact pre-metric-Bourbaki algebroid such that $\chi$ is $g$-isotropic. Then, the $E$-metric $g$ is of the form
\begin{equation}
    g \left( (\phi \oplus \chi)(U + z), (\phi \oplus \chi)(V + y) \right) = g_S(U + z, V + y) + F(U, V),
\label{ei13}
\end{equation}
where $F$ is an $R$-valued symmetric $(0, 2)$-type $(T(M))$-tensor, and the bracket $[\cdot,\cdot]_E$ is of the form
\begin{equation}
    [(\phi \oplus \chi)(U + z), (\phi \oplus \chi)(V + y)]_E = (\phi \oplus \chi)([U + z, V + y]_S) + \chi H(U, V),
\label{ei14}
\end{equation}
where $H: \mathfrak{X}(M) \times \mathfrak{X}(M) \to \mathfrak{X}(Z)$ is a map satisfying
\begin{align}
    H(U, f V) &= f H(U, V), \nonumber\\
    H(f U, V) &= f H(U, V) + l_{\underline{d} f} F(U, V),
\label{ei15}
\end{align}
for all $U, V \in \mathfrak{X}(M), f \in C^{\infty}(M, \mathbb{R})$.
\label{ti2}
\end{Theorem}

\begin{proof} For the $E$-metric, we can directly see that $F(U, V)$ is given by $g(\phi(U), \phi(V))$ as $\chi$ is $g$-isotropic. It is $C^{\infty}(M, \mathbb{R})$-bilinear and takes values in $R$, so that it is an $R$-valued $(0, 2)$-type tensor, which is also symmetric due to the symmetry of $g$. 

Recall that the $g$-isotropy of $\chi$ implies that $g(\chi(z), u) = \iota_{\rho_E(u)} z$ for all $z \in \mathfrak{X}(Z), u \in \mathfrak{X}(E)$, which will be frequently used in the discussion below.

For the bracket, we decompose it into four terms: 
\begin{equation} [\phi(U), \phi(V)]_E, \qquad [\phi(U), \chi(y)]_E, \qquad [\chi(z), \phi(V)]_E, \qquad [\chi(z), \chi(y)]_E. \nonumber
\end{equation}
For the term $[\phi(U), \chi(y)]_E$, we use the metric invariance property by choosing $u = \phi(U), v = \chi(y)$ with some arbitrary $w$. The left-hand side of the metric invariance property is given by
\begin{align} 
LHS &= \mathcal{L}^R_{\rho_E \phi(U)}(g(\chi(y), w)) = \mathcal{L}^R_U \iota_{\rho_E(w)} y \nonumber\\
&= \iota_{[U, \rho_E(w)]} y + \iota_{\rho_E(w)} \mathcal{L}^Z_U y \nonumber\\
&= \iota_{[U, \rho_E(w)]} y + g(\chi(\mathcal{L}^Z_U y), w). \nonumber
\end{align}
The right-hand side of the metric invariance property is given by
\begin{align}
    RHS &= g([\phi(U), \chi(y)]_E, w) + g(\chi(y), [\phi(U), w]_E) \nonumber\\
    &= g([\phi(U), \chi(y)]_E, w) + \iota_{\rho_E([\phi(U), w]_E)} y \nonumber\\
    &= g([\phi(U), \chi(y)]_E, w) + \iota_{[\rho_E \phi(U), \rho_E(w)]} y \nonumber\\
    &= g([\phi(U), \chi(y)]_E, w) + \iota_{[U, \rho_E(w)]} y. \nonumber
\end{align}
After equating these two, the non-degeneracy implies
\begin{equation} [\phi(U), \chi(y)]_E = \chi(\mathcal{L}^Z_U y). \nonumber
\end{equation}
For the term $[\chi(z), \phi(V)]_E$, we use the symmetric part of the bracket together with this result and get
\begin{align} [\chi(z), \phi(V)]_E &= - [\phi(V), \chi(z)]_E + \chi d g(\phi(V), \chi(z)) \nonumber\\
&= \chi \left( - \mathcal{L}^Z_V z + d \iota_V z \right). \nonumber
\end{align}
For the term $[\chi(z), \chi(y)]_E$, we use the metric invariance property by choosing $u = \chi(z), v = \chi(y)$ with some arbitrary $w$. By exactness and the $\mathbb{R}$-linearity of $\mathcal{L}^R$, the left-hand side of the metric invariance property vanishes, and the right-hand side of it becomes
\begin{align}
    RHS &= g([\chi(z), \chi(y)]_E, w) + g(\chi(y), [\chi(z), w]_E) \nonumber\\
    &= g([\chi(z), \chi(y)]_E, w) + \iota_{\rho_E([\chi(z), w]_E} y \nonumber\\
    &= g([\chi(z), \chi(y)]_E, w), \nonumber
\end{align}
where the last equality follow from exactness and the fact that the anchor $\rho_E$ is a morphism of brackets. As the left-hand side of the metric invariance property vanishes, the non-degeneracy of the $E$-metric $g$ implies that $[\chi(z), \chi(y)]_E = 0$ for all $z, y \in \mathfrak{X}(Z)$, so that this term does not contribute.

For the term $[\phi(U), \phi(V)]_E$, we consider the difference $\Delta(U, V) := [\phi(U), \phi(V)]_E - \phi([U, V])$. We have
\begin{align}
    \Delta(U, f V) &= [\phi(U), \phi(f V)]_E - \phi([U, f V]) \nonumber\\
    &= [\phi(U), f \phi(V)]_E - \phi(f [U, V] + U(f) V) \nonumber\\
    &= f [\phi(U), \phi(V)]_E + \rho_E \phi(U)(f) \phi(V) - f \phi([U, V]) + U(f) \phi(V) \nonumber\\
    &= f \Delta(U, V), \nonumber
\end{align}
by the right-Leibniz rule. Similarly, we can evaluate
\begin{align}
    \Delta(f U, V) &= [\phi(f U), \phi(V)]_E - \phi([f U, V]) \nonumber\\
    &= [f \phi(U), \phi(V)]_E - \phi(f [U, V] - V(f) U) \nonumber\\
    &= f [\phi(U), \phi(V)]_E - \rho_E \phi(V)(f) \phi(U) + \chi l_{\underline{d} f} g(\phi(U), \phi(V))- f \phi([U, V]) + V(f) \phi(U) \nonumber\\
    &= f \Delta(U, V) + \chi l_{\underline{d} f} g(\phi(U), \phi(V)), \nonumber\\
    &= f \Delta(U, V) + \chi l_{\underline{d} f} F(U, V), \nonumber
\end{align}
by the left-Leibniz rule. Note that $\Delta$ is a map of the form $\mathfrak{X}(M) \times \mathfrak{X}(M) \to \mathfrak{X}(E)$; its range is not in $\mathfrak{X}(Z)$. Yet, now we will show that the image of $\Delta$ is in the image of $\chi$. By exactness, it is equivalent to show that the image of $\Delta$ is in the kernel of the anchor $\rho_E$:
\begin{align} 
\rho_E(\Delta(U, V)) &= \rho_E([\phi(U), \phi(V)]_E - \phi([U, V])) \nonumber\\
&= [\rho_E \phi(U), \rho_E \phi(V)] - \rho_E \phi([U, V]) \nonumber\\
&= [U, V] - [U, V] = 0,\nonumber
\end{align}
so that we have $im(\Delta) \subset im(\chi)$. Consequently, we can find a map $H: \mathfrak{X}(M) \times \mathfrak{X}(M) \to \mathfrak{X}(Z)$ such that $\chi H = \Delta$ and it satisfies Equation (\ref{ei15}). Combining all of these results, we get the desired conclusion.
\end{proof}

\begin{Corollary} In the same setting as Theorem (\ref{ti2}), if there is a $g$-isotropic splitting $\phi$, then $F = 0$, and $H$ is a $Z$-valued 2-form.
\label{ci1}
\end{Corollary}

\begin{proof} If $\phi$ is $g$-isotropic, then $F = 0$. When $F$ = 0, $H$ is a $Z$-valued $(0, 2)$-type tensor by Equation (\ref{ei15}). Moreover, it is anti-symmetric because one has
\begin{equation}
    H(U, V) + H(V, U) = d g(\phi(U), \phi(V)) = d F(U, V) = 0, \nonumber
    \end{equation}
so that $H$ is a $Z$-valued 2-form.
\end{proof}


\section{Examples}
\label{s10}

\noindent In this section, we focus on some algebroids from the literature that are special cases of almost-Bourbaki or almost-metric-Bourbaki algebroids. We give the definitions of these algebroids in a suitable manner to our own notation. Hence, this section together with Section (\ref{s12}) also serve as a dictionary for various algebroid structures presented in a unified notation. We also present the induced Bourbaki pre-calculus for possible cases. When there is another Bourbaki pre-calculus, related but not the \textit{induced} one of Definition (\ref{di2}), we also point out these constructions, and the relations between Bourbaki algebroids and Bourbaki pre-calculi. 

As our hierarchical construction can be used to define new algebroids, satisfying a different combination of our founding axioms, in principle we can easily talk about algebroids such as pre-metric-Lie, almost-dull, pre-Jacobi algebroids, or vector bundle valued higher Courant algebroids. For example, a pre-metric-Lie algebroid $E$ would be a pre-Lie algebroid equipped with an $R$-valued $E$-metric that is invariant in the sense of (\ref{eg7}) for some arbitrary vector bundle $R$. When they are straightforward, we will sometimes omit these definitions for a clear presentation.

\subsection{Lie Algebroids}

Any almost-Lie algebroid $(E, \rho_E, [\cdot,\cdot]_E)$ can be considered as an almost-Bourbaki algebroid with
\begin{equation}
(E, \rho_E, [\cdot,\cdot]_E, R, g, \mathbb{D} = 0, \mathbb{L} = 0),
\label{ej1}
\end{equation}
where $R$ and $g$ are somewhat arbitrary as $\mathbb{D}$ is zero. Consequently, every pre-Lie and Lie algebroid is a pre-Bourbaki and Bourbaki algebroid of the same form, respectively. Moreover, $H$-twisted Lie algebroids \cite{23} are pre-Bourbaki algebroids that satisfy a Jacobi identity ``twisted'' by a $ker(\rho_E)$-valued $E$-3-form $H$. 

For an exact pre-metric-Lie algebroid $(E, \rho_E, [\cdot,\cdot]_E, R, g, Z, \chi, d = 0, l = 0, \mathcal{L}^R)$, the induced Bourbaki pre-calculus is given by
\begin{equation}
    (\iota, d = 0, l = 0, \mathcal{L}^R, \mathcal{L}^Z = \nabla),
\label{ej2}
\end{equation}
for a usual vector bundle connection $\nabla: \mathfrak{X}(M) \times \mathfrak{X}(Z) \to \mathfrak{X}(Z)$. Consequently, the standard bracket is of the form
\begin{equation}
    [U + z, V + y]_S = [U, V] + \nabla_U y - \nabla_V z.
\label{ej3}
\end{equation}
Note that this induced Bourbaki pre-calculus is not the usual Cartan calculus even though the latter is constructed from the Lie bracket.

\subsection{Dull Algebroids}

\begin{Definition} \cite{24} A triplet $(E, \rho_E, [\cdot,\cdot]_E)$ is called an almost-dull algebroid if it is a local almost-Leibniz algebroid whose locality operator is zero.
\label{dy1}
\end{Definition}
By Proposition (\ref{pf1}), $L = 0$ implies that the symmetric part $S$ of the bracket $[\cdot,\cdot]_E$ is $C^{\infty}(M, \mathbb{R})$-bilinear. Coincidentally, the decomposition $S = \mathbb{D} g$ implies that $D$ is also $C^{\infty}(M, \mathbb{R})$-linear, meaning that it is a zeroth-order differential operator. Hence, any almost-dull algebroid is an almost-Bourbaki algebroid of the form
\begin{equation} (E, \rho_E, [\cdot,\cdot]_E, R, g, \mathbb{D}, \mathbb{L} = 0),
\label{ej5}
\end{equation}
for some arbitrary vector bundle $R$ and $R$-valued $E$-metric $g$, where $\mathbb{D}$ is a zeroth-order differential operator. Moreover, any pre-dull and dull algebroid is a pre-Bourbaki and Bourbaki algebroid, respectively.

For an exact pre-dull algebroid $(E, \rho_E, R, g, Z, \chi, d, l = 0, \mathcal{L}^R)$, where $d$ is a zeroth-order differential operator, the induced Bourbaki pre-calculus is given by
\begin{equation} (\iota, d, l = 0, \mathcal{L}^R, \mathcal{L}^Z = \nabla),
\label{ej6}
\end{equation}
for a usual vector bundle connection $\nabla$ on $Z$. This is parallel to exact pre-metric-Lie algebroid case, where the latter is a special case of the former as $0$ is a zeroth-order differential operator. 

\subsection{Metric and Courant Algebroids}

Any almost-Courant algebroid $(E, \rho_E, [\cdot,\cdot]_E, g)$ is an almost-Bourbaki algebroid with
\begin{equation}
(E, \rho_E, [\cdot,\cdot]_E, R = \Lambda^0(T^*(M)), g, \mathbb{D} = g^{-1} D_E = g^{-1} \rho_E^* \underline{d}, \mathbb{L}_{\omega}(f) = g^{-1} \rho_E^*(\omega f)).
\label{ej7}
\end{equation}
As the $E$-metric for almost-Courant algebroids takes values in $C^{\infty}(M, \mathbb{R})$, one has $\mathbb{L}_{\omega}(f) = g^{-1} \rho_E^*(\omega f) = f g^{-1} \rho_E^*(\omega)$. On the other hand, any metric algebroid $(E, \rho_E, [\cdot,\cdot]_E, g)$ is an almost-metric-Bourbaki algebroid with

\begin{equation}
(E, \rho_E, [\cdot,\cdot]_E, R = \Lambda^0(T^*(M)), g, \mathbb{D} = g^{-1} D_E = g^{-1} \rho_E^* \underline{d}, \mathbb{L}_{\omega}(f) = g^{-1} \rho_E^*(\omega f), \mathcal{L}^R = \underline{\mathcal{L}}).
\label{ej8}
\end{equation}
Consequently, any pre-Courant algebroid both in the sense of \cite{14} and of Definition (\ref{dd4}) is a pre-metric-Bourbaki algebroid, and any Courant algebroid is a metric-Bourbaki algebroid. Additionally, $H$-twisted Courant algebroids \cite{25}, and $\mathcal{T}$-twisted Courant algebroids \cite{26} are also pre-metric-Bourbaki algebroids of the same form. Moreover, exact versions of almost-Courant algebroids can be considered as exact almost-Bourbaki algebroids of the form
\begin{equation}
(E, \rho_E, [\cdot,\cdot]_E, R = \Lambda^0(T^*(M)), g, Z = T^*(M), \chi = g^{-1} \rho_E^*, d = \underline{d}, l_{\omega}(f) = f \omega).
\label{ej9}
\end{equation}
For the pair $R = \Lambda^0(T^*(M)), Z = \Lambda^1(T^*(M))$, the induced Bourbaki pre-calculus is given by,
\begin{equation} (\iota = \underline{\iota}, d = \underline{d}, l_{\omega} f = f \omega, \mathcal{L}^R = \underline{\mathcal{L}}, \mathcal{L}^Z = \underline{\mathcal{L}}).
\label{ej10}
\end{equation}
Moreover, our main results reproduce \v{S}evera classification partially and Theorem 5.6 of \cite{27} about exact $\mathcal{T}$-twisted Courant algebroids where isotropic splittings exist in these cases.

As the Leibniz-Jacobi identity is only used for showing that $H$ is closed in \v{S}evera classification, one can immediately conclude the same equality (\ref{ed14}) or (\ref{ed15}) for exact pre-Courant algebroids with $H$ being a 3-form which is not necessarily closed. Moreover, in the proof, one first gets a $T^*(M)$-valued 2-form $H$, but by using the metric invariance property (\ref{ed10}), one concludes that $H$ is indeed a 3-form.

Even though we only consider real and finite rank vector bundles, some complex infinite dimensional vector bundles also satisfy properties similar to a Bourbaki algebroid, namely exotic Courant algebroids \cite{28}, which is not surprising since the purpose of exotic ones is to generalize Courant algebroids in some certain context.


\subsection{Higher Metric and Higher Courant Algebroids}

Similarly to the previous section, any higher almost-Courant $(E, \rho_E, [\cdot,\cdot]_E, g, \chi)$ is an almost-Bourbaki algebroid 
\begin{equation}
(E, \rho_E, [\cdot,\cdot]_E, R = \Lambda^{k-1}(T^*(M)), g, \mathbb{D} = \chi \underline{d}, \mathbb{L}_{\omega}(\eta) = \chi(\omega \wedge \eta)),
\label{ej11}
\end{equation}
for some $k$. On the other hand, any higher metric algebroid $(E, \rho_E, [\cdot,\cdot]_E, g, \chi)$ is an almost-metric-Bourbaki algebroid with
\begin{equation}
(E, \rho_E, [\cdot,\cdot]_E, R = \Lambda^{k-1}(T^*(M)), g, \mathbb{D} = \chi \underline{d}, \mathbb{L}_{\omega}(\eta) = \chi(\omega \wedge \eta), \mathcal{L}^R = \underline{\mathcal{L}}).
\label{ej12}
\end{equation}
Again consequently, every higher pre-Courant algebroid is a pre-metric-Bourbaki algebroid, and every higher Courant algebroid is a Bourbaki algebroid. Moreover, exact versions of higher almost- Courant algebroids can be considered as exact almost-Bourbaki algebroids of the form
\begin{equation}
(E, \rho_E, [\cdot,\cdot]_E, R = \Lambda^{k-1}(T^*(M)), g, Z = \Lambda^k(T^*(M)), \chi, d = \underline{d}, l_{\omega}(\eta) = \omega \wedge \eta).
\label{ej13}
\end{equation}
For the pair $R = \Lambda^{k-1}(T^*(M)), Z = \Lambda^k(T^*(M)$, the induced Bourbaki pre-calculus is given by,
\begin{equation} (\iota = \underline{\iota}, d = \underline{d}, l_{\omega} \eta = \omega \wedge \eta, \mathcal{L}^R = \underline{\mathcal{L}}, \mathcal{L}^Z = \underline{\mathcal{L}}).
\label{ej14}
\end{equation}
We can conclude that the usual Cartan calculus is the Bourbaki pre-calculus induced by exact pre-metric and exact higher pre-metric algebroids.

For the standard higher Courant algebroid $\mathbb{T}^k(M)$, in both higher Courant algebroid \cite{3} and Vinogradov Lie $k$-algebroid perspectives \cite{29}, it is noted that the bracket can be twisted by a $(k+2)$-form. Our results indicates that this twist can be done by $\Lambda^{k}(T^*(M))$-valued 2-forms, which include all $(k+2)$-forms.

If one wishes to define higher versions of $H$-twisted or $T$-twisted Courant algebroids, they would also fit into pre-metric-Bourbaki algebroids.

\subsection{Cartan Calculus on Vector Bundle Valued Forms}

There is a straightforward generalization of Cartan calculus on vector bundle valued forms, provided the vector bundle is endowed with a connection. For a vector bundle $P$, a $P$-valued $k$-form can be considered as a $C^{\infty}(M, \mathbb{R})$-multilinear anti-symmetric map from the $k$th tensor power of the tangent bundle to the vector bundle $P$, and the set of all $P$-valued $k$-forms is denoted by $\Omega^k(M; P) := \Gamma(Hom(\Lambda^k(T(M)), P))$. If $P$ is endowed with a usual vector bundle connection $\nabla: \mathfrak{X}(M) \times \mathfrak{X}(P) \to \mathfrak{X}(P)$, then we can define the exterior covariant derivative $\underline{d}^\nabla: \Omega^k(M; P) \to \Omega^{k+1}(M; P)$ such that
\begin{align}
(\underline{d}^{\nabla} \omega)(V_1, \ldots, V_{k+1}) &:= \sum_{1 \leq i \leq k+1} (-1)^{i+1} \nabla_{V_i}(\omega(V_1, \ldots, \check{V_i}, \ldots, V_{k+1})) \nonumber\\
& \quad \ + \sum_{1 \leq i < j \leq {k+1}} \omega([V_i, V_j], V_1, \ldots, \check{V_i}, \ldots \check{V_j}, \ldots, V_{k+1}),
\label{ej15}
\end{align}
for $V_i \in \mathfrak{X}(M)$, which is parallel to Equation (\ref{eb3}). Analogously, we have the interior product $\underline{\iota}^{\nabla}: \Omega^{k+1}(M; P) \to \Omega^k(M; P)$\footnote{It does not depend on $\nabla$, but we denote it with a $\nabla$ to differentiate it from other interior products. We will do this for other interior products, too.}. Similarly, we can define the Lie covariant derivative $\underline{\mathcal{L}}^{\nabla}: \mathfrak{X}(M) \times \Omega^k(M; P) \to \Omega^k(M; P)$ such that
\begin{equation} 
(\underline{\mathcal{L}}^{\nabla}_V \omega)(V_1, \ldots, V_k) := \nabla_V(\omega(V_1, \ldots, V_k)) - \sum_{i=1}^k \omega(V_1, \ldots, \underline{\mathcal{L}}_V V_i, \ldots V_k),
\label{ej16}
\end{equation}
which is parallel to Equation (\ref{eb1}). These maps constitute a Bourbaki pre-calculus, as they satisfy Cartan magic formula and $(\underline{\iota}^{\nabla})^2 = 0$, just as the usual versions.

A natural question is that from which pre-metric-Bourbaki algebroid this Bourbaki pre-calculus is induced. Even though such algebroid structures do not exist to our knowledge, we can make an educated guess in our framework. This leads us to define vector bundle valued higher Courant algebroids. Yet, we will only present the following definition; the others would be defined analogously.

\begin{Definition} A septet $(E, \rho_E, [\cdot,\cdot]_E, P, g, \chi, \nabla)$ is called a $P$-valued higher almost-Courant algebroid if
\begin{itemize}
    \item $(E, \rho_E, [\cdot,\cdot]_E)$ is an almost-Leibniz algebroid,
    \item $P$ is a vector bundle,
    \item $g$ is a $Hom(\Lambda^{k-1}(T(M)), P)$-valued $E$-metric,
    \item $\chi: Hom(\Lambda^k(T(M)), P) \to E$ is a vector bundle morphism,
    \item $\nabla$ is a vector bundle connection on $P$,
    \item The symmetric part of the bracket satisfies
    \begin{equation} [u, v]_E + [v, u]_E = \chi \underline{d}^{\nabla} g(u, v),
    \label{ej17}
    \end{equation}
    for all $u, v \in \mathfrak{X}(R)$.
\end{itemize}
\label{dj2}
\end{Definition}
\noindent Exact $P$-valued higher pre-Courant algebroids would yield the Cartan calculus on $P$-valued forms as the induced Bourbaki pre-calculus.

\subsection{Jacobi Algebroids}

Jacobi algebroids offers a good toy model for the discussion of the previous section as they are defined as Lie algebroids equipped with an $E$-connection. 
\begin{Definition} \cite{30, 31} A quintet $(E, \rho_E, [\cdot,\cdot]_E, R, \nabla)$ is called a Jacobi algebroid if
\begin{itemize}
    \item $(E, \rho_E, [\cdot,\cdot]_E)$ is a Lie algebroid,
    \item $R$ is a line bundle,
    \item $\nabla$ is a flat $E$-connection on $R$, i.e.
    \begin{equation}
        \nabla_u \nabla_v r - \nabla_v \nabla_u r - \nabla_{[u, v]_E} r = 0,
    \label{ej18}
    \end{equation}
    for all $u, v \in \mathfrak{X}(E), r \in \mathfrak{X}(R)$.
\end{itemize}
\label{dj3}
\end{Definition}
As any Jacobi algebroid is a Lie algebroid by definition, we can see it as a Bourbaki algebroid of the form (\ref{ej1}). Even though there is no metric or metric invariance property, a Jacobi algebroid is endowed with an $E$-connection, which can be considered as a $C^{\infty}(M, \mathbb{R})$-bilinear map $\nabla: E \to \mathfrak{D}(R)$. The flatness condition (\ref{ej18}) implies that this map is actually a morphism of brackets from $E$ to $\mathfrak{D}(R)$, which is itself a Lie algebroid with the usual commutator. 

In order to define the exterior covariant derivative or Lie covariant derivative on $R$-valued forms, we need a usual vector bundle connection on $R$, not an $E$-connection. A $T(M)$-connection $\tilde{\nabla}$ is said to be \textit{induced} by an $E$-connection $\nabla$ if $\tilde{\nabla} (\rho_E \otimes id_R) = \nabla$, i.e. $\tilde{\nabla}_{\rho_E(u)} r = \nabla_u r$, for all $u \in \mathfrak{X}(E), r \in \mathfrak{X}(R)$. Given an $E$-connection, if we can find a corresponding induced $T(M)$-connection, then we can construct the exterior covariant derivative and Lie covariant derivative. Yet, there is no reason to have a well-defined, unique induced $T(M)$-connection for every $E$-connection. However, for the conformal Courant algebroids of the next subsection, this is indeed the case, and one can make use of it.

\subsection{Conformal Courant Algebroids}

\begin{Definition} \cite{32} A conformal Courant algebroid is a sextet $(E, \rho_E, [\cdot,\cdot]_E, R, g, \nabla)$ such that
\begin{itemize}
    \item $(E, \rho_E, [\cdot,\cdot]_E)$ is a Leibniz algebroid,
    \item $R$ is a line bundle,
    \item $g$ is an $R$-valued $E$-metric,
    \item $\nabla$ is an $E$-connection on $R$,
    \item The symmetric part of the bracket satisfies\footnote{Here, we consider $g^{-1}$ as a map $\mathfrak{X}(R) \times \mathfrak{X}(E^*) \to \mathfrak{X}(E)$ and $\nabla$ as a map $\mathfrak{X}(R) \to \mathfrak{X}(E^*) \times \mathfrak{X}(R)$.}
    \begin{equation} [u, v]_E + [v, u]_E = g^{-1} \nabla g(u, v),
    \label{ej19}
    \end{equation}
    \item The following metric invariance property holds
    \begin{equation} \nabla_u (g(v, w)) = g([u, v]_E, w) + g(v, [u, w]_E),
    \label{ej20}
    \end{equation}
\end{itemize}
for all $u, v, w \in \mathfrak{X}(E)$.
\label{dj4}
\end{Definition}
By Lemma 2.1 of \cite{32}, $\nabla$ is flat in the same sense as Equation (\ref{ej18}), so that conformal Courant algebroids are similar to Jacobi algebroids in this sense, but they have a bracket which has a symmetric part.

By Lemma 2.4 of \cite{32}, for the connection $\nabla$, there is a unique and well-defined induced $T(M)$-connection $\tilde{\nabla}$. By using this uniquely defined $\tilde{\nabla}$, we can replace $\nabla_u$ with $\tilde{\nabla}_{\rho_E(u)}$ in the left-hand side of Equation (\ref{ej20}). Hence, any conformal Courant algebroid is a metric-Bourbaki algebroid of the form
\begin{equation} (E, \rho_E, [\cdot,\cdot]_E, R, g, \mathbb{D} = g^{-1} \nabla, \mathbb{L}_{\omega}(r) = g^{-1}(\rho_E^*(\omega) \otimes r), \mathcal{L}^R = \tilde{\nabla}).
\label{ej21}
\end{equation}
Moreover, by using the induced $T(M)$-connection $\tilde{\nabla}$, we can construct exterior covariant derivative $\underline{d}^{\tilde{\nabla}}$, Lie covariant derivative $\underline{\mathcal{L}}^{\nabla}$ and interior product $\underline{\iota}^{\nabla}$ acting on $R$-valued $k$-forms, which would yield a Bourbaki pre-calculus. 

A conformal Courant algebroid $(E, \rho_E, [\cdot,\cdot]_E, R, g, \nabla)$ is said to be exact if the following is an exact sequence of vector bundles \cite{32}: 
\begin{equation}
    0 \xrightarrow{\quad} T^*(M) \otimes R \xrightarrow{\quad g^{-1} (\rho_E^* \otimes id_R) \quad} E \xrightarrow{\ \quad \rho_E \ \quad} T(M) \xrightarrow{\quad} 0.
\label{ej22}
\end{equation}
Any such exact conformal Courant can be seen as an exact metric-Bourbaki algebroid of the form
\begin{equation} (E, \rho_E, [\cdot,\cdot]_E, R, g, Z = T^*(M) \otimes R, \chi = g^{-1} (\rho_E^* \otimes id_R), d = \tilde{\nabla}, l_{\omega} r = \omega \otimes r, \mathcal{L}^R = \tilde{\nabla}),
\label{ej23}
\end{equation}
as the action on $\underline{d}^{\tilde{\nabla}}$ and $\tilde{\nabla}$ agrees on $R$-valued $0$-forms, i.e. on sections of $R$. The induced Bourbaki pre-calculus coincides with the Cartan calculus on $R$-valued $k$-forms constructed from $\tilde{\nabla}$, and our main results reproduce partially Propositions 2.6, 2.7 and 2.8 in \cite{32}. Every exact conformal Courant algebroid admits a $g$-isotropic splitting by Lemma 2.5 of \cite{32}, so that the twist is done by a $(T^*(M) \otimes L)$-valued 2-form, which becomes an $L$-valued 3-form. 



\section{$A$-Bourbaki Algebroids and Bourbaki $A$-Pre-Calculus}
\label{s11}

\noindent In previous sections, we constructed Bourbaki and metric-Bourbaki algebroids together with Bourbaki pre-calculi. In Section (\ref{s10}), we saw that these algebroids are general enough so that many algebroids from the literature fit into this framework. Yet, there is one more straightforward generalization by replacing the tangent Lie algebroid with an arbitrary pre-Lie algebroid\footnote{One can also try to replace the tangent Lie algebroid with an exact pre-Bourbaki algebroid, yet this construction will be out of the scope of this paper.}. With this modification, in the next section, we will see that we can cover a lot more examples from the algebroid literature. 

This natural generalization follows from the fact that the choice of the exterior derivative $\underline{d}$ in the symbol map, and consequently in $\mathbb{L}_{\underline{d} f}$, is somewhat arbitrary. Similar considerations to Equation (\ref{ef3}) also apply here; for consistency, the map that acts on the smooth functions should be a derivation over $C^{\infty}(M, \mathbb{R})$ like the exterior derivative. Hence, instead, one can define an ``$A$-symbol map'' $\sigma^A$ for an anchored vector bundle $(A, \rho_A)$:
\begin{align} 
\sigma^A: Dif^1(R, E) &\to Hom(A^*, Hom(R, E)) \nonumber\\
D &\mapsto \left( D_A f \mapsto \left( r \mapsto D(f r) - f D(r) \right) \right),
\label{ek1}
\end{align}
where $D_A := \rho_A^* \underline{d}: C^{\infty}(M, \mathbb{R}) \to \mathfrak{X}(A^*)$. With the help of this $A$-symbol map and with the insight from $AV$-Courant algebroids \cite{6}, we can define $A$-Bourbaki algebroids simply as Bourbaki algebroids whose $D$ is equipped with an $A$-symbol $\mathbb{L}^A$ instead of the usual $T(M)$-symbol $\mathbb{L}$. That is, we have $\mathbb{D}: \mathfrak{X}(R) \to \mathfrak{X}(Z)$ and $C^{\infty}(M, \mathbb{R})$-multilinear $\mathbb{L}^A: \Omega^1(A) \times \mathfrak{X}(R) \to \mathfrak{X}(Z)$ satisfying $\mathbb{D}(f r) = f \mathbb{D} r + \mathbb{L}^A_{D_A f} r$ for all $f \in C^{\infty}(M, \mathbb{R}), r \in \mathfrak{X}(R)$. With this observation, we can replace the tangent bundle $T(M)$ with the anchored vector bundle $A$ in our constructions of Sections (\ref{s6} - \ref{s9}), and everything still works with possible some extra conditions. Hence, we present the $A$-versions of our previous constructions and results in a compact manner.
\begin{Definition} A nonet $(E, \rho_E, [\cdot,\cdot]_E, R, g, A, \rho_A, \mathbb{D}, \mathbb{L}^A)$ is called an $A$-almost-Bourbaki algebroid if
\begin{itemize}
    \item $(A, \rho_A)$ is an anchored vector bundle,
    \item $(E, \rho_E, [\cdot,\cdot]_E, R, g, \mathbb{D}, \mathbb{L}^A)$ is an almost-Bourbaki algebroid, but $\mathbb{D}$ is a first-order differential operator from $R$ to $Z$ equipped with the $A$-symbol $\mathbb{L}^A$.
\end{itemize}
An $A$-almost-Bourbaki algebroid is called $A$-pre-Bouraki (resp. $A$-Bourbaki) algebroid if it is also a pre-Leibniz (resp. Leibniz) algebroid.
\label{dk1}
\end{Definition}

\begin{Definition} A dectet $(E, \rho_E, [\cdot,\cdot]_E, R, g, A, \rho_A, \mathbb{D}, \mathbb{L}^A, \mathcal{L}^R)$ is called an $A$-almost-metric-Bourbaki\footnote{Behold, triple hyphens!} algebroid if
\begin{itemize}
    \item $(A, \rho_A)$ is an anchored vector bundle,
    \item $(E, \rho_E, [\cdot,\cdot]_E, R, g, \mathbb{D}, \mathbb{L}^A, \mathcal{L}^R)$ is an almost-metric-Bourbaki algebroid, but $\mathbb{D}$ is a first-order differential operator from $R$ to $Z$ equipped with the $A$-symbol $\mathbb{L}^A$.
\end{itemize}
An $A$-almost-metric-Bourbaki algebroid is called $A$-pre-metric-Bouraki (resp. $A$-metric-Bourbaki) algebroid if it is also a pre-Leibniz (resp. Leibniz) algebroid.
\label{dk2}
\end{Definition}

\begin{Definition} A duodectet $(E, \rho_E, [\cdot,\cdot]_E, R, g, Z, \chi, A, \rho_A, d, l^A, Q)$ is called an exact $A$-almost-Bourbaki algebroid if
\begin{itemize}
    \item $Z$ is a vector bundle,
    \item $\chi: Z \to E$ is a vector bundle morphism,
    \item $(E, \rho_E, [\cdot,\cdot]_E, R, g, A, \rho_A, \mathbb{D} = \chi d, \mathbb{L}^A = \chi l^A)$ is an $A$-almost-Bourbaki algebroid,
    \item $Q: E \to A$ is a vector bundle morphism,
    \item The following is an exact sequence of vector bundles
    \begin{equation}
        0 \xrightarrow{\quad} Z \xrightarrow{\quad \chi \quad} E \xrightarrow{\ \quad Q \ \quad} A \xrightarrow{\quad} 0.
    \label{ek2}
    \end{equation}
\end{itemize}
A tredectet $(E, \rho_E, [\cdot,\cdot]_E, R, g, Z, \chi, A, \rho_A, d, l^A, \mathcal{L}^R, Q)$ is called an exact $A$-almost-metric-Bourbaki algebroid if $Z, \chi$ and $Q$ are exactly as above, and $(E, \rho_E, [\cdot,\cdot]_E, R, g, A, \rho_A, \mathbb{D} = \chi d, \mathbb{L}^A = \chi l^A, \mathcal{L}^R)$ is an $A$-almost-metric algebroid.
\label{dk3}
\end{Definition}
Furthermore, if $A$ is also endowed with an almost-Lie algebroid structure $(A, \rho_A, [\cdot,\cdot]_A)$, then we can define a generalization of Bourbaki pre-calculus.
\begin{Definition} Let $(R, Z)$ be a pair of vector bundles, and $(A, \rho_A, [\cdot,\cdot]_A)$ an almost-Lie algebroid. The quintet $(\iota, d, l^A, \mathcal{L}^R, \mathcal{L}^Z)$ is called a Bourbaki $A$-pre-calculus if 
\begin{itemize}
\item $\iota: \mathfrak{X}(A) \times \mathfrak{X}(Z) \to \mathfrak{X}(R)$ is a $C^{\infty}(M, \mathbb{R})$-bilinear map,
\item $d: \mathfrak{X}(R) \to \mathfrak{X}(Z)$ is a first-order differential operator with the $A$-symbol $l^A$,
\item $\mathcal{L}^R: \mathfrak{X}(A) \times \mathfrak{X}(R) \to \mathfrak{X}(R)$ is an $\mathbb{R}$-bilinear map,
\item $\mathcal{L}^Z: \mathfrak{X}(A) \times \mathfrak{X}(Z) \to \mathfrak{X}(Z)$ is a first-order differential operator satisfying
    \begin{enumerate}
        \item $\mathcal{L}^Z_a(f z) = f \mathcal{L}^Z_a z + \rho_A(a)(f) z$,
        \item $\mathcal{L}^Z_{f a} z = f \mathcal{L}^Z_a z + l^A_{D_A f} \iota_a z$,
    \end{enumerate}
\item They satisfy
\begin{align}
    &\mathcal{L}^R_a \iota_b z = \iota_{[a, b]_A} z + \iota_b \mathcal{L}^Z_a z, \nonumber\\
    &\iota_a \left( \mathcal{L}^Z_b z - d \iota_b z \right) = - \iota_b \left( \mathcal{L}^Z_a z - d \iota_a z \right)
\end{align}
\end{itemize}
for all $f \in C^{\infty}(M, \mathbb{R}), a, b \in \mathfrak{X}(A), z \in \mathfrak{X}(Z)$.
\label{dk4}
\end{Definition}
\noindent Any exact $A$-pre-metric-Bourbaki algebroid induces an Bourbaki $A$-pre-calculus completely analogously to the usual version, since Lemma (\ref{li1}) and Proposition (\ref{pi1}) still hold. 
\begin{Definition} Let $(E, \rho_E, [\cdot,\cdot]_E, R, g, Z, \chi, A, \rho_A, d, l^A, \mathcal{L}^R)$ be an exact $A$-pre-metric-Bourbaki algebroid such that $\chi$ is $g$-isotropic. The Bourbaki pre-calculus $(\iota, d, l^A, \mathcal{L}^R, \mathcal{L}^Z)$, where $\iota$ and $\mathcal{L}^Z$ are defined analogously to Equation (\ref{ei7}) and Equation (\ref{ei6}), is called the ``induced'' Bourbaki $A$-pre-calculus by $E$.
\label{dk5}
\end{Definition}
If one considers the tangent bundle $T(M)$ as an anchored vector bundle with the identity map, then $T(M)$-almost-Bourbaki and $T(M)$-almost-metric-Bourbaki algebroids become almost-Bourbaki and almost-metric-Bourbaki algebroids, respectively. Moreover, in addition if we take the map $Q$ as the anchor of $E$, then exact $T(M)$-almost-Bourbaki and exact $T(M)$-almost-metric-Bourbaki algebroids and Bourbaki $T(M)$-pre-calculus yield the Bourbaki versions. Furthermore, under analogous assumptions, with possibly a small extra condition on $Q$, our main results still hold. 
\begin{Theorem} Let $(\iota, d, l^A, \mathcal{L}^R, \mathcal{L}^Z)$ be a Bourbaki $A$-pre-calculus on the pair of vector bundles $(R, Z)$ for a pre-Lie algebroid $(A, \rho_A, [\cdot,\cdot]_A)$, then $A \oplus Z$ becomes an exact $A$-pre-metric-Bourbaki algebroid
\begin{equation}
    (E = A \oplus Z, \rho_E = \rho_A Q = \rho_A proj_1, [\cdot,\cdot]_S, R, g_S, Z, \chi = inc_2, A, \rho_A, d, l^A, \mathcal{L}^R, Q = proj_1),
\label{ek3}
\end{equation}
where the ``standard bracket'' $[\cdot,\cdot]_S$ is given by
\begin{equation} 
    [a + z, b + y]_S := [a, b]_A + \mathcal{L}^Z_a y - \mathcal{L}^Z_b z + d \iota_b z,
\label{ek4}
\end{equation}
and the standard $R$-valued $E$-metric $g_S$ is defined as 
\begin{equation}
    g_S(a + z, b + y) := \iota_a y + \iota_b z.
\label{ek5}
\end{equation}
Here, $proj_1: E = A \oplus Z \to T(M)$ is the projection onto the first component, and $inc_2: Z \to E = A \oplus Z$ is the canonical inclusion. 
\label{tk1}
\end{Theorem}

\begin{proof} The proof is completely parallel to the proof of Theorem (\ref{ti1}). 
\end{proof}
\noindent Note that as $Q = proj_1$ and the range of both $\mathcal{L}$ and $d$ is $Z$, the map $Q$ is a morphism of brackets, i.e. we have
\begin{equation} Q([u, v]_E) = [Q(u), Q(v)]_A,
\label{ek6}
\end{equation}
for all $u, v \in \mathfrak{X}(E)$.
\begin{Theorem}
Let $(\iota, d, l^A, \mathcal{L}^R, \mathcal{L}^Z)$ be a Bourbaki pre-calculus for a pre-Lie algebroid $(A, \rho_A, [\cdot,\cdot]_A)$, and $(E, \rho_E, [\cdot,\cdot]_E, R, g, Z, \chi, A, Q, d, l^A, \mathcal{L}^R)$ an exact $A$-pre-metric-Bourbaki algebroid such that $\chi$ is $g$-isotropic and $Q$ is a morphism of brackets. Then, for a splitting $\phi$ of the exact sequence (\ref{ek2}), the $E$-metric $g$ is of the form
\begin{equation}
    g \left( (\phi \oplus \chi)(a + z), (\phi \oplus \chi)(b + y) \right) = g_S(a + z, b + y) + F(a, b),
\label{ek7}
\end{equation}
where $F$ is an $R$-valued symmetric $(0, 2)$-type $A$-tensor, and the bracket $[\cdot,\cdot]_E$ is of the form
\begin{equation}
    [(\phi \oplus \chi)(a + z), (\phi \oplus \chi)(b + y)]_E = (\phi \oplus \chi)([a + z, b + y]_S) + \chi H(a, b),
\label{ek8}
\end{equation}
where $H: \mathfrak{X}(A) \times \mathfrak{X}(A) \to \mathfrak{X}(Z)$ is a map satisfying
\begin{align}
    H(a, f b) &= f H(a, b), \nonumber\\
    H(f a, b) &= f H(a, b) + l^A_{D_A f} F(a, b),
\label{ek9}
\end{align}
for all $a, b \in \mathfrak{X}(M), f \in C^{\infty}(M, \mathbb{R})$.
\label{tk2}
\end{Theorem}

\begin{proof} The proof is completely parallel to the proof of Theorem (\ref{ti2}). The only different point is that we used the fact that $Q$ is a morphism of brackets.
\end{proof}

\begin{Corollary} In the same setting as Theorem (\ref{tk2}), if there is a $g$-isotropic splitting $\phi$, then $F = 0$, and $H$ is a $Z$-valued $A$-2-form.
\label{ck1}
\end{Corollary}

\begin{proof} The proof is completely parallel to the proof of Corollary (\ref{ci1}).
\end{proof}


\section{More Examples}
\label{s12}

\noindent In this section, we present several examples of $A$-Bourbaki algebroids, $A$-metric-Bourbaki algebroids and Bourbaki $A$-pre-calculi from the literature.

\subsection{Cartan Calculus on Lie Algebroids}

For any almost-Lie algebroid $(A, \rho_A, [\cdot,\cdot]_A)$, one can construct an exterior derivative, Lie derivative and interior product on $A$-forms completely analogously to the usual Cartan versions \cite{33}. These operators yield an Bourbaki $A$-pre-calculus. Yet they are not the induced ones from an $A$-pre-metric-Lie algebroid, where in this case the induced Bourbaki $A$-pre-calculus would be $(\iota, d = 0, l^A = 0, \mathcal{L}^R, \mathcal{L}^Z = \nabla)$ for an $A$-connection on $Z$. On the other hand, if one considers an exact $A$-pre-metric-Bourbaki algebroid with the pair $R = \Lambda^{k-1}(A^*), Z = \Lambda^k(A^*)$ and construct a standard bracket on $A \oplus \Lambda^k(A^*)$ from these analogous maps, then the induced Bourbaki $A$-pre-calculus coincide with the Cartan calculus on $A$, analogously to the higher exact pre-metric algebroids. 

In particular, when the base manifold $M$ carries a Poisson structure, then the cotangent bundle $T^*(M)$ carries a Lie algebroid structure induced by the Poisson bivector. This Lie algebroid yields a $T^*(M)$-Bourbaki pre-calculus on multi-vector fields \cite{34}. 

For an almost-dull algebroid, the induced Bourbaki $A$-pre-calculus similarly would be $(\iota, d, l^A = 0, \mathcal{L}^R, \mathcal{L}^Z = \nabla)$ for an $A$-connection on $Z$, where $d$ is a zeroth-order differential operator.

\subsection{Cartan Calculus on Vector Bundle Valued $A$-forms}
\label{ss121}

\noindent Cartan calculus on almost-Lie algebroids can be generalized to vector bundle valued $A$-forms, completely parallel to the usual case. Let $P$ be a vector bundle, $(A, \rho_A, [\cdot,\cdot]_A)$ an almost-Lie algebroid, $\nabla$ an $A$-connection on $P$\footnote{The $A$-connection $\nabla$ on $P$ can be considered as a vector bundle morphism $\nabla: A \to \mathfrak{D}(P)$, so that sometimes it is said that $P$ is an $A$-module \cite{21} or $P$ carries a representation of $A$.}, then analogously to Equation (\ref{ej15}) and (\ref{ej16}), one can define an exterior covariant derivative $d^{\nabla}$, Lie covariant derivative $\mathcal{L}^{\nabla}$ and interior product $\iota^{\nabla}$ on $P$-valued $A$-forms. These maps yield a Bourbaki $A$-pre-calculus by Proposition 7.1.3 of \cite{21}. If we modify Definition (\ref{dj2}) of $P$-valued higher almost-Courant algebroids with $Hom(\Lambda^k(A), P)$ instead of $Hom(\Lambda^k(T(M)), P)$, we would get an algebroid structure such that when exact pre-versions are considered, the induced Bourbaki $A$-pre-calculus would be the Cartan calculus on $P$-valued $A$-forms.

\subsection{$AV$-Courant Algebroids}

By using the language of Cartan calculus on $V$-valued $A$-forms of Subsection (\ref{ss121}), one can define $AV$-Courant algebroids.
\begin{Definition} \cite{6} A dectet $(E, [\cdot,\cdot]_E, A, \rho_A, [\cdot,\cdot]_A, Q, V, \nabla, g, \chi)$ is called an $A V$-Courant algebroid if
\begin{itemize}
    \item $(\mathfrak{X}(E), [\cdot,\cdot]_E)$ is a Leibniz algebra,
    \item $(A, \rho_A, [\cdot,\cdot]_A)$ is a Lie algebroid, 
        \item $Q: E \to A$ is a vector bundle morphism satisfying $Q([u, v]_E) = [Q(u), Q(v)]_A$,
    \item $V$ is a vector bundle,
    \item $\nabla$ is an $A$-connection on $V$\footnote{It can be considered as a vector bundle morphism $\nabla: A \to \mathfrak{D}(V)$.},
    \item $g$ is a $V$-valued $E$-metric,
    \item $\chi: V \otimes A^* \to E$ is a vector bundle morphism such that
    \begin{equation} g(\chi(\xi), u) = \iota^{\nabla}_{Q(u)} \xi,
    \label{el5}
    \end{equation}
    for all $\xi \in \mathfrak{X}(V \otimes A^*) = \Omega^1(A; V)$,
    \item The symmetric part of the bracket is given by
    \begin{equation} [u, v]_E + [v, u]_A = \chi d^{\nabla} g(u, v),
    \label{el6}
    \end{equation}
    for all $u, v \in \mathfrak{X}(E)$,
    \item The following metric invariance property holds
    \begin{equation} \mathcal{L}^{\nabla}_{Q(u)}(g(v, w)) = g([u, v]_E, w) + g(v, [u, w]_E),
    \label{el7}
    \end{equation}
    for all $u, v, w \in \mathfrak{X}(E)$,
    \item The following is an exact sequence of vector bundles
    \begin{equation}
        0 \xrightarrow{\quad} V \otimes A^* \xrightarrow{\quad \chi \quad} E \xrightarrow{\ \quad Q \ \quad} A \xrightarrow{\quad} 0.
    \label{el8}
    \end{equation}
\end{itemize}
\label{dl2}
\end{Definition}
\noindent Presented in this manner, it is clear that any $AV$-Courant algebroid is an exact $A$-metric-Bourbaki algebroid of the form
\begin{equation}
    (E, \rho_E = \rho_A Q, [\cdot,\cdot]_E, R = V, g, Z = V \otimes A^*, \chi, A, \rho_A, d = d^{\nabla}, l^A, \mathcal{L}^R = \mathcal{L}^{\nabla}, Q),
\label{el9}
\end{equation}
as one has 
\begin{equation} [u, f v]_E = f [u, v]_E + (\rho_A Q(u))(f) v,
\label{el10}
\end{equation}
for all $u, v \in \mathfrak{X}(E), f \in C^{\infty}(M, \mathbb{R})$. For $A V$-Courant algebroids, Equation (\ref{el5}) implies that $\chi$ is $g$-isotropic, so our main results partially reproduce Proposition 1 and Theorem 1 of \cite{6}. Moreover, $g$-isotropic splittings exist, so that the twist of the standard bracket is given by a $Z$-valued $A$-2-form. The induced Bourbaki $A$-pre-calculus coincides with the Cartan calculus on $V$-valued $A$-forms on the pair $R = V, Z = V \otimes A^*$.

\subsection{$G$-Algebroids}

\begin{Definition} \cite{4} For a Lie group $G$, consider a principal $G$-bundle for an admissible group data set\footnote{See \cite{4}, for details.} and two associated vector bundles $E$ and $R$. A septet $(E, \rho_E, [\cdot,\cdot]_E, R, g, \mathbb{D}, \mathbb{L}^E)$ is called a $G$-algebroid if
\begin{itemize}
    \item $(E, \rho_E, [\cdot,\cdot]_E)$ is a Leibniz algebroid,
    \item $R$ is a vector bundle, 
    \item $g$ is an $R$-valued $E$-metric,
    \item $\mathbb{L}^E: \Omega^1(E) \times \mathfrak{X}(R) \to \mathfrak{X}(E)$ is a $C^{\infty}(M, \mathbb{R})$-bilinear map,
    \item $D: \mathfrak{X}(R) \to \mathfrak{X}(E)$ is an $\mathbb{R}$-linear map satisfying
    \begin{equation}
        \mathbb{D}(f r) = f \mathbb{D} r + \mathbb{L}^E_{D_E f} r,
    \label{el11}
    \end{equation}
    \item The symmetric part of the bracket satisfies
    \begin{equation}
        [u, v]_E + [v, u]_E = \mathbb{D} g(u, v),
    \label{el12}
    \end{equation}
    \item The action $[u,\cdot]_E$ preserves the $G$-structure,
\end{itemize}
for all $f \in C^{\infty}(M, \mathbb{R}), r \in \mathfrak{X}(R), u, v \in \mathfrak{X}(E)$.
\label{dl3}
\end{Definition}
\noindent When one forgets about the information about the principal $G$-bundle and structure preserving, it is clear that the associated bundle $E$ carries an $E$-Bourbaki algebroid structure. Even though this definition is quite similar to our Bourbaki algebroid definition, their exact versions differ; where in \cite{4} exact ones fit into the sequence
\begin{equation}
    T^*(M) \otimes R \xrightarrow{\quad  \quad} E \xrightarrow{\ \quad \rho_E \ \quad} T(M) \xrightarrow{\quad} 0,
\label{el13}
\end{equation}
and there is no restriction on $\mathbb{D}$ or $\mathbb{L}^E$ analogous to Equation (\ref{eh1}) or (\ref{eh2}).

\subsection{Omni-Lie, Higher Omni-Lie Algebroids}

One can find another way to generalize the standard Courant algebroid $T(M) \oplus T^*(M)$, when one considers $T(M) = \mathfrak{D}(\Lambda^0(T^*(M))$ and $T^*(M)$ as its dual. If one instead considers $\mathfrak{D}(P)$ for some arbitrary vector bundle $P$, then one can construct another vector bundle $\mathfrak{J}(P)$, which is isomorphic to the first jet bundle $\mathfrak{J}^1(P)$ and is dual to $\mathfrak{D}(P)$ in the sense that is explained in \cite{7}\footnote{The non-degeneracy condition is dealt with this $P$-dual notion.}. In short, we replace $\Lambda^0(T^*(M)), T(M)$ and $T^*(M)$ with $P, \mathfrak{D}(P)$ and $\mathfrak{J}(P)$, respectively. Following the constructions of \cite{7}, we have maps
\begin{align}
\mathfrak{d}:& \ \mathfrak{X}(P) \to \mathfrak{X}(\mathfrak{J}(P)), \nonumber\\
\mathfrak{L}: & \ \mathfrak{X}(\mathfrak{D}(P)) \times \mathfrak{X}(\mathfrak{J}(P)) \to \mathfrak{X}(\mathfrak{J}(P)), \nonumber\\
\langle \cdot,\cdot \rangle_P: & \ \mathfrak{X}(\mathfrak{D}(P)) \times \mathfrak{X}(\mathfrak{J}(P)) \to \mathfrak{X}(P).
\label{el14}
\end{align}
These maps gives a Bourbaki $\mathfrak{D}(P)$-pre-calculus. If one considers the vector bundle $E = \mathfrak{D}(P) \oplus \mathfrak{J}(P)$, one can construct a bracket on $E$,
\begin{equation} [D_1 + j_1, D_2 + j_2]_E := [D_1, D_2]_{\mathfrak{D}(P)} + \mathfrak{L}_{D_1} j_2 - \mathfrak{L}_{D_2} j_1 + \mathfrak{d} \langle D_2, j_1 \rangle_P,
\label{el15}
\end{equation}
which is exactly of the same form as our standard bracket (\ref{ek4}).
\begin{Definition} \cite{7} A quartet $(E, [\cdot,\cdot]_E, g, Q)$ is called an omni-Lie algebroid if
\begin{itemize}
    \item $E = \mathfrak{D}(P) \oplus \mathfrak{J}(P)$ for some vector bundle $P$,
    \item The bracket is given as in Equation (\ref{el15}),
    \item $g(D_1 + j_1, D_2 + j_2) := \langle D_1, j_2 \rangle_P + \langle D_2, j_1 \rangle_P$,
    \item $Q: E = \mathfrak{D}(P) \oplus \mathfrak{J}(P) \to \mathfrak{D}(P)$ is the projection onto the first component.
\end{itemize}
\label{dl4}
\end{Definition}
\noindent By Theorem 3.1 of \cite{7}, every omni-Lie algebroid is an exact $\mathfrak{D}(P)$-metric-Bourbaki algebroid of the form
\begin{align} &(E = \mathfrak{D}(P) \oplus \mathfrak{J}(P), \rho_E = \rho_A Q = \rho_{\mathfrak{D}(P)} proj_1, [\cdot,\cdot]_E, R = P, g, Z = \mathfrak{J}(P), \chi = inc_2, \nonumber\\
& \qquad \qquad \qquad A = \mathfrak{D}(P), \rho_A = \rho_{\mathfrak{D}(P)}, d = \mathfrak{d}, l^A, \mathcal{L}^R = Q = proj_1, Q = proj_1),
\label{el16}
\end{align}
which satisfies $Q([u, v]_E) = [Q(u), Q(v)]_{\mathfrak{D}(P)}$. The induced Bourbaki $\mathfrak{D}(P)$-pre-calculus is given by the maps in Equation (\ref{el14}) for the pair $R = P, Z = \mathfrak{J}(P)$.

A natural question is that whether there is an analogous \textit{higher} version of omni-Lie algebroids, and the answer is affirmative \cite{35}. One can extend the maps $\mathfrak{d}, \mathfrak{L}$ and $\langle \cdot,\cdot \rangle_P$ to $k$th jet bundle $\mathfrak{J}^k(P)$, so that these maps coincide with the Cartan calculus on $P$-valued $\mathfrak{D}(P)$-$k$-forms, $Hom(\Lambda^k(\mathfrak{D}(P)), P)$, which does not require any connection on $P$ since $\mathfrak{D}(P)$ is already the derivation bundle \cite{35}. Consequently, one can define a bracket of the form (\ref{el15}), which makes $E = \mathfrak{D}(P) \oplus \mathfrak{J}^k(P)$ an exact $\mathfrak{D}(P)$-pre-metric-Bourbaki algebroid of the form
\begin{align} &(E = \mathfrak{D}(P) \oplus \mathfrak{J}^k(P), \rho_E = \rho_A Q = \rho_{\mathfrak{D}(P)} proj_1, [\cdot,\cdot]_E, R = \mathfrak{J}^{k-1}(P), g, Z = \mathfrak{J}^k(P), \chi = inc_2, \nonumber\\
& \qquad \qquad \qquad A = \mathfrak{D}(P), \rho_A = \rho_{\mathfrak{D}(P)}, d = \mathfrak{d}, l^A, \mathcal{L}^R = Q = proj_1, Q = proj_1),
\label{el17}
\end{align}
by Theorem 2.11 of \cite{35}, which is called a higher omni-Lie algebroid. It also satisfies $Q([u, v]_E) = [Q(u), Q(v)]_{\mathfrak{D}(P)}$. The induced Bourbaki $\mathfrak{D}(P)$-pre-calculus is given by the maps in Equation (\ref{el14}) for the pair $R = \mathfrak{J}^{k-1}(P), Z = \mathfrak{J}^k(P)$. According to Remark 2.13 of \cite{35}, one can deform the bracket of a higher omni-Lie algebroid by a section of $\mathfrak{J}^{k+2}(P)$, which would be a special case for a twist by a $\mathfrak{J}^k(P)$-valued $\mathfrak{D}(P)$-2-form. For $k = 1$, a higher omni-Lie algebroid reduces to an omni-Lie algebroid. 

There is another possible generalization of omni-Lie algebroids, namely $E$-Courant algebroids \cite{5}, they are also metric-Bourbaki algebroids, and our main results are directly related to Theorem 5.4 of \cite{5}.


\subsection{Anti-Commutable Pre-Leibniz Algebroids}

Another way to construct a Bourbaki $A$-pre-calculus is to consider a local almost-Leibniz algebroid $E$ and admissible linear $E$-connections, which are defined in \cite{18}. On a local almost-Leibniz algebroid $(E, \rho_E, [\cdot,\cdot]_E, L^E)$ given by a locality operator of the form $L^E: \Omega^1(E) \times \mathfrak{X}(E) \times \mathfrak{X}(E) \to \mathfrak{X}(E)$, a linear $E$-connection $\nabla$ is called ``admissible'' if the symmetric part of the bracket satisfies
\begin{equation}
    [u, v]_E + [v, u]_E = L^E(e^a, \nabla_{X_a} u, v) + L^E(e^a, \nabla_{X_a} v, u),
\label{ez101}
\end{equation}
for all $u, v \in \mathfrak{X}(E)$, where $(X_a)$ is a local $E$-frame with its dual local $E$-coframe $(e^a)$. If such linear $E$-connections exist, then the algebroid is said to be \textit{anti-commutable}. In \cite{18}, it is proven that for an admissible linear $E$-connection $\nabla$, the following \textit{modified bracket} is an almost-Lie bracket:
\begin{equation} [u, v]_E^{\nabla} := [u, v]_E - L^E(e^a, \nabla_{X_a} u, v).
\label{ez102}
\end{equation}
The modified bracket enters the definition of the $E$-torsion operator since the usual form of torsion is not tensorial when it is constructed with a local almost-Leibniz bracket. Accordingly, in order to have a tensorial torsion, one defines \cite{17} $E$-torsion operator
\begin{align} T(\nabla)(u, v) &:= \nabla_u v - \nabla_v u - [u, v]_E + L^E(e^a, \nabla_{X_a} u, v) \nonumber\\
& \ = \nabla_u v - \nabla_v u - [u, v]_E^{\nabla}.
\label{ez103}
\end{align}
The tensoriality of the $E$-curvature operator is a little bit more involved, and it includes locality projectors defined in \cite{13}. On a regular local almost-Leibniz algebroid $(E, \rho_E, [\cdot,\cdot]_E, L^E)$, a locality projector is defined as a $C^{\infty}(M, \mathbb{R})$-linear map $\mathcal{P}: \mathfrak{X}(E) \to \mathfrak{X}(E)$ satisfying $im(\mathcal{P} L^E) \subset ker(\rho_E)$, and $\mathcal{P}|_{ker(\rho_E)} = id_{ker(\rho_E)}$. Given a locality projector $\mathcal{P}$ on a regular local pre-Leibniz algebroid $E$, the following \textit{projected modified bracket} is a pre-Lie bracket:
\begin{equation} [u, v]_E^{\hat{\nabla}} := [u, v]_E - \mathcal{P} L^E(e^a, \nabla_{X_a} u, v),
\label{ez104}
\end{equation}
for an admissible linear $E$-connection $\nabla$. In terms of this new bracket one defines \cite{18} a tensorial $E$-curvature operator:
\begin{align} R(\nabla)(u, v)w &:= \nabla_u \nabla_v w - \nabla_v \nabla_u w - \nabla_{[u, v]_E} w + \nabla_{\mathcal{P} L^E(e^a, \nabla_{X_a} u, v)} \nonumber\\
& \ = \nabla_u \nabla_v w - \nabla_v \nabla_u w - \nabla_{[u, v]_E^{\hat{\nabla}}}.
\label{ez105}
\end{align}

One can then extend the brackets $[\cdot,\cdot]_E^{\nabla}$ and $[\cdot,\cdot]_E^{\hat{\nabla}}$ to \textit{$E$-Leibniz derivatives} $\mathcal{L}(\nabla), \hat{\mathcal{L}}(\nabla): \mathfrak{X}(E) \times \Omega^k(E) \to \Omega^k(E)$, analogously to Equation (\ref{eb1}). Moreover, one can define a \textit{modified and projected modified $E$-exterior derivatives}\footnote{These are not exterior covariant derivatives or generalizations of them.} $d(\nabla), \hat{d}(\nabla): \Omega^k(E) \to \Omega^{k+1}(E)$ analogously to Equation (\ref{eb3}) \cite{18}. Coincidentally, we have two Bourbaki $E$-pre-calculus for an admissible linear connection $\nabla$ of an anti-commutable regular pre-Leibniz algebroid on the vector bundle pair $R = \Lambda^{k-1}(E^*), Z = \Lambda^k(E^*)$, 
\begin{align} & \Bigl( \iota = \iota^E, d = d(\nabla), l_{\Omega} \Upsilon = \Omega \wedge \Upsilon, \mathcal{L}^R = \mathcal{L}(\nabla), \mathcal{L}^Z = \mathcal{L}(\nabla) \Bigr), \nonumber\\
& \left( \iota = \iota^E, d = \hat{d}(\nabla), l_{\Omega} \Upsilon = \Omega \wedge \Upsilon, \mathcal{L}^R = \hat{\mathcal{L}}(\nabla), \mathcal{L}^Z = \hat{\mathcal{L}}(\nabla) \right).
\label{ez106}
\end{align}

We conclude this section with the admissible linear $E$-connections on almost-Bourbaki algebroids. By definition of admissible linear $E$-connections and almost-Bourbaki algebroids, we have
\begin{equation} [u, v]_E + [v, u]_E = \mathbb{D} g(u, v) = L^E(e^a, \nabla_{X_a} u, v) + L^E(e^a, \nabla_{X_a} v, u).
\label{ez107}
\end{equation}
By Corollary (\ref{cf1}), we have $L^E(\Omega, u, v) = \mathbb{L}^E_{\Omega}(g(u, v))$, where $\mathbb{L}^E$ is the $E$-symbol of $\mathbb{D}$. Hence, the condition for admissibility becomes
\begin{equation}
    \mathbb{D}(g(u, v)) = \mathbb{L}^E_{e^a} \left( g(\nabla_{X_a} u, v) + g(u, \nabla_{X_a} v) \right),
\label{ez108}
\end{equation}
which resembles a form of metric-compatibility condition. In \cite{18}, it was proven that for several algebroid structures, the admissibility is indeed equivalent to metric-compatibility in some generalized sense. In particular, it was shown that for almost-Courant algebroids, the condition is equivalent to usual $E$-metric-compatibility. The above result (\ref{ez108}) reproduces all of these for particular choices of $\mathbb{L}^E$. 

We also note that for exact pre-Bourbaki algebroids, there is no need for a locality projector when defining tensorial curvature operators since the image of $L^E$ already lies inside the image of $\chi$, which is equal to kernel of the anchor by the exactness.


\section{Concluding Remarks}
\label{s13}

\noindent In this paper, we analyze the hierarchy of defining axioms of Courant algebroids: almost-Leibniz, almost-Courant, metric, pre-Courant and Courant. Then, we repeat the same construction for the higher case, where we axiomatize the properties of standard higher Courant algebroids constructed with the higher Dorfman bracket \cite{3}. Focusing on the crucial points of these structures, we imitate another hierarchy, starting from an almost-Leibniz algebroid. The next step is the construction of almost-Bourbaki algebroids, where the symmetric part of the bracket is given by a vector bundle valued metric and a first-order differential operator. Then, by introducing an additional data, a map acting on the sections of an arbitrary vector bundle that would be analogous to Lie derivative of forms, we generalize the notion of metric invariance and define almost-metric-Bourbaki algebroids, parallel to metric algebroids. By considering the cases when the anchor is a morphism of brackets and the Leibniz-Jacobi is satisfied, we define the final steps in our hierarchy, namely pre-metric-Bourbaki and metric-Bourbaki algebroids. Similarly to exact Courant algebroids, we define exact versions of these algebroids. For the case of exact pre-metric-Bourbaki algebroids, we induce some maps generalizing the exterior derivative, Lie derivative, interior product and taking a wedge product with a 1-form. Axiomatizing the important properties of these maps with the insight from the \v{S}evera classification \cite{16}, we define Bourbaki pre-calculi. Conversely, given a Bourbaki pre-calculus, we show that one can construct a Dorfman-like bracket in order to get an exact pre-metric-Bourbaki algebroid. Moreover, we prove that any exact pre-metric-Bourbaki algebroid with certain additional conditions has to have a bracket that is the twisted version of this Dorfman-like bracket. We establish the properties of this twist, which is a result that partly generalizes the \v{S}evera classification. Later, we further generalize our constructions, by replacing the tangent bundle with an arbitrary pre-Lie algebroid $A$. We define $A$-Bourbaki algebroids, $A$-metric-Bourbaki algebroids, Bourbaki $A$-pre-calculi, and extend our results for them. Moreover, we give several examples of Bourbaki and metric-Bourbaki algebroids together with their $A$-versions, so that our paper also serves as a dictionary of various algebroid structures from the literature with a unified notation. These examples include Lie, dull \cite{24}, metric \cite{14}, Courant, higher metric, higher Courant \cite{3}, Jacobi \cite{30}, Conformal Courant \cite{32} algebroids for Bourbaki algebroids, and $AV$-Courant \cite{6}, $G$- \cite{4}, omni-Lie \cite{7}, Higher omni-Lie \cite{35} algebroids for $A$-Bourbaki algebroids. We also give several examples of Bourbaki pre-calculi including Cartan calculus of usual forms, of vector bundle valued forms, of Lie algebroids forms, and of vector bundle valued Lie algebroid forms, and the induced ones from the exact pre-metric-Bourbaki algebroid examples. Moreover, we explain the relations between Bourbaki algebroids and Bourbaki pre-calculi for these specific cases, and use our machinery to construct various new algebroid structures.

A natural continuation of our work would be to define $A$-Bourbaki algebroids and Bourbaki $A$-pre-calculus, when $A$ is itself an exact Bourbaki algebroid. Yet, these constructions seem non-trivial, as one has to consider complicated combinations of vector bundle valued metrics and locality operators in a way that resembles some sort of semi-direct product. We plan to investigate this case in detail in the future. We hope that these results will lead to a ``doubled'' version of Bourbaki algebroids. This construction would be analogous to the Drinfeld double of Lie bialgebroids \cite{36}, and might answer the question ``what is the double of a Courant algebroid?'' \cite{37}. Lie bialgebroids and Courant algebroids equipped with Dirac structures are in one-to-one correspondence, so these results might be extended to Courant or Bourbaki versions of bialgebroids. Closely related to this topic, there is another correspondence between Lie bialgebroids and mathced pair of Lie algebroids \cite{38}, where the latter consist of two Lie algebroids equipped with flat connections satisfying certain compatibility conditions. These discussions seem to be closely related to the possible generalizations of Bourbaki algebroids and Bourbaki pre-calculi, and we plan to work on the details of these correspondences in a more general setting in the near future.


It would be also interesting to  weaken some of the assumptions in our main structures and results. Even though we try to make the constructions in the most general way, we still have restrictive assumptions. For example, the anchor is a morphism of brackets, the image of the map $\chi$ is $g$-isotropic, and the map $\mathcal{L}^R_{\rho_E(u)}$ replaces\footnote{See footnote 11.} $\mathcal{L}^R_u$. On the other direction with stronger assumptions, we never consider the effect of Leibniz-Jacobi identity on the conditions of Bourbaki pre-calculus in this paper. We have some preliminary but promising results in this direction, and we will investigate this route in a consecutive paper.

Another possible route for future work is to consider a module/derivation structure over some arbitrary ring instead of the ring of smooth functions $C^{\infty}(M, \mathbb{R})$ in line with Lie-Rinehart algebras \cite{39}. We expect this generalization to work without any serious difficulties. Moreover, there are many geometrical structures that can be ``lifted'' to the level of pre-Leibniz algebroids, and in particular of Bourbaki algebroids, including complex, symplectic, contact and Finsler structures. For example, in \cite{19}, statistical and Hessian structures are generalized in the framework of pre-Leibniz algebroids. We plan to work on these various structures on Bourbaki algebroids in the near future. Currently, we are working on the generalizations of conformal and projective structures, which will hopefully lead us to the algebroid version of tractor geometry.


\section*{Acknowledgments}

We are all thankful to Cem Yetişmişoğlu, Yasin \c{C}akmak and Se\c{c}il Tunal\i \ for valuable discussions on many details of this work. We are also grateful to Oğul Esen for pointing out a possible research route by bringing the literature on matched pairs of Lie algebroids to our attention in a QDIS conference. Hopefully, this route will yield a fruitful collaboration in the future.

KD is funded by İstanbul Technical University BAP Postdoctoral Research Fellowship (DOSAP) with project number ``TAB-2021-43202''.


\newpage



\begin{thebibliography}{50}
\bibitem{1} O. Hohm, C. Hull, \& B. Zwiebach, Generalized metric formulation of double field theory. \textit{Journal of High Energy Physics, 8} (2010); \href{https://arxiv.org/abs/1006.4823}{arXiv:1006.4823 [hep-th]}.
\bibitem{2} N. Hitchin, Generalized Calabi-Yau manifolds. \textit{The Quarterly Journal of Mathematics, 54}, 281–308 (2003); \href{https://arxiv.org/abs/0209099}{arXiv:0209099 [math.DG]}.
\bibitem{3} Y. Bi, \& Y. Sheng, On higher analogues of Courant algebroids. \textit{Science China Mathematics 54}, 437–447 (2011); \href{https://arxiv.org/abs/1003.1350}{arXiv:1003.1350 [math.DG]}.
\bibitem{4} M. Bugden, O. Hulík, F. Valach, \& D. Waldram, G-algebroids: a unified framework for exceptional and generalised geometry, and Poisson–Lie duality. \textit{Fortschritte der Physik 69}, 1-11 (2021); \href{https://arxiv.org/abs/2103.01139v2}{arXiv:2103.01139 [math.DG]}.
\bibitem{5} Z. Chen, Z. Liu, \& Y. Sheng, $E$-Courant algebroids. \textit{International Mathematics Research Notices 2010}, 4334-4376 (2010); \href{https://arxiv.org/abs/0805.4093v3}{arXiv:0805.4093 [math.DG]}.
\bibitem{6} D. Li-Bland, $AV$-Courant algebroids and generalized CR structures. \textit{Canadian Journal of Mathematics 63}, 938-960 (2011); \href{https://arxiv.org/abs/0811.4470}{arXiv:0811.4470 [math.DG]}.
\bibitem{7} Z. Chen, \& Z. J. Liu, Omni-Lie algebroids. \textit{Journal of Geometry and Physics 60}, 799-808 (2010); \href{https://arxiv.org/abs/0710.1923v1}{arXiv:0710.1923 [math-ph]}.
\bibitem{8} M. Boucetta, Riemannian geometry of Lie algebroids. \textit{Journal of the Egyptian Mathematical Society, 19}, 57-70 (2011); \href{https://arxiv.org/abs/0806.3522}{arXiv:0806.3522 [math.DG]}.
\bibitem{9} Y. Kosmann-Schwarzbach, \& K. C. H. Mackenzie, Differential operators and actions of Lie algebroids. \textit{Proceedings of Quantization, Poisson Brackets and Beyond}, (2002); \href{https://arxiv.org/abs/math/0209337}{arXiv:math/0209337 [math.DG]}.
\bibitem{10} N. Bourbaki, \textit{Alg\'{e}bre}. Hermann, Paris (1970).
\bibitem{11} A. Weinstein, Omni-Lie algebras. \textit{Berkeley CPAM Preprint Series, 774}, (1999); \href{https://arxiv.org/abs/math/9912190}{arXiv:math/9912190 [math.RT]}.
\bibitem{12} R. J. Fernandes, Lie algebroids, holonomy and characteristic classes. \textit{Advances in Mathematics, 170 }, 119-179 (2002); \href{https://arxiv.org/abs/math/0007132}{arXiv:math/0007132[math.DG]}.
\bibitem{13} T. Dereli, \& K. Doğan, Metric-connection geometries on pre-Leibniz algebroids: A search for geometrical structure in string models. \textit{Journal of Mathematical Physics 62}, (2021); \href{https://arxiv.org/abs/2006.05957}{arXiv:2006.05957 [hep-th]}.
\bibitem{14} I. Vaisman, Transitive Courant algebroids. \textit{International Journal of Mathematics and Mathematical Sciences 2005}, 1737-1758 (2005); \href{https://arxiv.org/abs/math/0407399}{arXiv:math/0407399 [math.DG]}.
\bibitem{15} A. J. Bruce, \& J. Grabowski, Pre-Courant algebroids. \textit{Journal of Geometry and Physics 142}, 254-273 (2019); \href{https://arxiv.org/abs/1608.01585}{arXiv:1608.01585 [math-ph]}.
\bibitem{16} P. \v{S}evera, Letters to Alan Weinstein about Courant algebroids. Unpublished manuscript (2017); \href{https://arxiv.org/abs/1707.00265}{arXiv:1707.00265 [math.DG]}.
\bibitem{17} B. Jur\u{c}o, \& J. Vysok\'{y}, Leibniz algebroids, generalized Bismut connections and Einstein-Hilbert actions. \textit{Journal of Geometry and Physics 97}, 25-33 (2015); \href{https://arxiv.org/abs/1503.03069}{arXiv:1503.03069 [hep-th]}.
\bibitem{18} T. Dereli, \& K. Doğan, `Anti-commutable' pre-Leibniz algebroids and admissible connections, (submitted) (2021); \href{https://arxiv.org/abs/2108.10199}{arXiv:2108.10199 [math.DG]}.
\bibitem{19} K. Doğan, `Statistical geometry and Hessian structures on pre-Leibniz algebroids, \textit{Journal of Physics: Conference Series 2191, 012011} (2022); \href{https://arxiv.org/abs/2109.03916}{arXiv:2109.03916 [math.DG]}.
\bibitem{20} J. Grabowski, D. Khudaverdyan, \& N. Poncin, The supergeometry of Loday algebroids. \textit{The Journal of Geometric Mechanics 5 }, 185-21 (2011); \href{https://arxiv.org/abs/1103.5852}{arXiv:1103.5852 [math.DG]}.
\bibitem{21} K. C. H. Mackenzie, \textit{General theory of Lie groupoids and Lie algebroids}. University Press, Cambridge (2005).
\bibitem{22} D. Baraglia, Leibniz algebroids, twistings and exceptional generalized geometry. \textit{Journal of Geometry and Physics 62 }, 903-934 (2012); \href{https://arxiv.org/abs/1101.0856}{arXiv:1101.0856 [math.DG]}.
\bibitem{23} M. Grützmann, $H$-twisted Lie algebroids. \textit{Journal of Geometry and Physics 61}, 476-484 (2011); \href{https://arxiv.org/abs/1005.5680}{arXiv:1005.5680 [math.DG]}.
\bibitem{24} M. J. Lean, Dorfman connections and Courant algebroids. \textit{Journal de Mathématiques Pures et Appliquées 116}, 1-39 (2018); \href{https://arxiv.org/abs/1209.6077v2}{arXiv:1209.6077 [math.DG]}.
\bibitem{25} M. Grützmann, $H$-twisted Courant algebroids. (2011); \href{https://arxiv.org/abs/1101.0993v2}{arXiv:1101.0993 [math.DG]}.
\bibitem{26} M. Hansen, \& T. Strobl, First class constrained systems and twisting of Courant algebroids by a closed 4-form. \textit{Fundamental Interactions: A Memorial Volume for Wolfgang Kummer, World Scientific}, 115–144 (2010); \href{https://arxiv.org/abs/0904.0711}{arXiv:0904.0711 [hep-th]}.
\bibitem{27} Y. Sheng, \& Z. Liu, Leibniz 2-algebras and twisted Courant algebroids. \textit{Communications in Algebra 41}, 1929-1953 (2013); \href{https://arxiv.org/abs/1012.5515}{arXiv:1012.5515 [math-ph]}.
\bibitem{28} J. Crilly, \& V. Mathai, Exotic Courant algebroids and T-duality. \textit{Journal of Geometry and Physics 163}, (2021); \href{https://arxiv.org/abs/1909.07127}{arXiv:1909.07127 [hep-th]}.
\bibitem{29} A. Deser, \& C. S\"{a}mann, Extended Riemannian geometry I: Local double field theory, \textit{Annales Henri Poincar\'e 19}, 2297–2346 (2018); \href{https://arxiv.org/abs/1611.02772}{arXiv:1611.02772 [hep-th]}.
\bibitem{30} A. G. Tortorella, The deformation $L_{\infty}$ algebra of a Dirac-Jacobi structure, (2021); \href{https://arxiv.org/abs/2111.07467}{arXiv:2111.07467 [math.DG]}.
\bibitem{31} J. Grabowski, \& G. Marmo, Jacobi structures revisited. \textit{Journal of Physics A: Mathematical and General 34}, 10975–10990 (2001); \href{https://arxiv.org/abs/math/0111148}{arXiv:math/0111148 [math.DG]}.
\bibitem{32} D. Baraglia, Conformal Courant algebroids and orientifold T-duality. \textit{International Journal of Geometric Methods in Modern Physics 10}, 1-35 (2013); \href{https://arxiv.org/abs/1109.0875v2}{arXiv:1109.0875 [math.DG]}.
\bibitem{33} C. Marle, Calculus on Lie algebroids, Lie groupoids and Poisson manifolds. \textit{Dissertationes Mathematicae 457}, 1-57 (2008); \href{https://arxiv.org/abs/0806.0919v3}{arXiv:0806.0919 [math.DG]}.
\bibitem{34} I. Vaisman, \textit{Lectures on the geometry of Poisson manifolds}. Birkhäuser, Basel (1994). 
\bibitem{35} Y. Bi, L. Vitagliano, \& T. Zhang, Higher omni-Lie algebroids. \textit{Journal of Lie Theory 29}, 881-899 (2019); \href{https://arxiv.org/abs/1812.09496v3}{arXiv:1812.09496 [math.DG]}.
\bibitem{36} K. C. H. Mackenzie, \& P. Xu, Lie bialgebroids and Poisson groupoids. \textit{Duke Mathematical Journal 73}, 415-452 (1984).
\bibitem{37} Z. Liu, A. Weinstein, \& P. Xu, Manin triples for Lie bialgebroids. \textit{Journal of Differential Geometry 45}, 547-674 (1997); \href{https://arxiv.org/abs/1812.09496v3}{arXiv:1812.09496 [math.DG]}.
\bibitem{38} T. Mokri, Matched pairs of Lie algebroids. \textit{Glasgow Mathematical Journal 39}, (2009).
\bibitem{39} G. S. Rinehart, Differential forms on general commutative algebras. \textit{Transactions of the American Mathematical Society 108}, 195-222 (1963).
\end{thebibliography}
\end{document}